\documentclass[12pt]{article}

\usepackage{amsfonts}

\usepackage{amscd}
\usepackage{amsthm}
\usepackage{indentfirst}

\usepackage[all]{xy}

\usepackage{amssymb, amsmath, amsthm, amsgen, amstext, amsbsy, amsopn}
\usepackage{epsfig}
\usepackage{epic,eepic}

\newtheorem{thm}{Theorem}[section]

\newtheorem{prop}[thm]{Proposition}
\newtheorem{cor}[thm]{Corollary}
\newtheorem{defin}[thm]{Definition}
\newtheorem{rem}[thm]{Remark}
\newtheorem{exam}[thm]{Example}

\newtheorem{question}[thm]{Question}

\newcommand{\R}{{\mathbb{R}}}

\newcommand{\N}{{\mathbb{N}}}

\newcommand{\cA}{{\mathcal{A}}}

\newcommand{\cD}{{\mathcal{D}}}
\newcommand{\cC}{{\mathcal{C}}}

\newcommand{\cF}{{\mathcal{F}}}
\newcommand{\cG}{{\mathcal{G}}}

\newcommand{\cL}{{\mathcal{L}}}

\newcommand{\cN}{{\mathcal{N}}}
\newcommand{\cP}{{\mathcal{P}}}

\newcommand{\cS}{{\mathcal{S}}}

\newcommand{\cU}{{\mathcal{U}}}
\newcommand{\cV}{{\mathcal{V}}}
\newcommand{\cW}{{\mathcal{W}}}

\def\id{{1\hskip-2.5pt{\rm l}}}

\newcommand{\Ham}{{\hbox{\it Ham\,}}}

\newcommand{\f}{{\vec{f}}}
\newcommand{\g}{{\vec{g}}}

\begin{document}

\title{Symplectic geometry of quantum noise}

\renewcommand{\thefootnote}{\alph{footnote}}

\author{\textsc Leonid
Polterovich }

\footnotetext[1]{ Partially supported by the National Science
Foundation grant DMS-1006610, the Israel Science Foundation grants 509/07, 178/13
and the European Research Council Advanced grant 338809.}

\date{November, 2013}

\maketitle

\medskip
\noindent
\begin{abstract}
\noindent We discuss a quantum counterpart, in the sense of the Berezin-Toeplitz quantization, of certain constraints on Poisson brackets coming from ``hard" symplectic geometry. It turns out that they can be interpreted in terms of the quantum noise of observables and their joint measurements in operational quantum mechanics. Our findings include various geometric mechanisms of quantum noise production and
a noise-localization uncertainty relation. The methods involve Floer theory and Poisson bracket invariants
originated in function theory on symplectic manifolds.
\end{abstract}

\vfill\eject

\tableofcontents

\vfill\eject

\section{Introduction}

The main theme of the present paper is an interaction between ``hard" symplectic geometry
and operational quantum mechanics, where the role of observables
is played by positive operator valued measures (POVMs) \cite{Busch}. We focus on POVMs coming from partitions of unity of a classical phase space, that is a closed symplectic manifold, under the Berezin-Toeplitz quantization. Such POVMs, considered earlier in \cite{P}, model a {\it registration procedure}, a statistical procedure which can be considered as an attempt to localize the system in the phase space and which is closely
related to approximate quantum measurements. Certain constraints on Poisson brackets coming from symplectic
geometry can be  interpreted in terms of the quantum noise of these observables and their joint measurements.
Our findings include various geometric mechanisms of quantum noise production and a noise-localization uncertainty relation. The methods involve Floer theory and Poisson bracket invariants
originated in function theory on symplectic manifolds.

\subsection{Registration procedure}\label{subsec-reg-intro}
For an open cover $\cU=\{U_1,...,U_L\}$ of a classical phase space $M$,
 the registration procedure yields an answer to the question {\it `Where (i.e. in which set $U_i$) is the system located?'} The ambiguity arising due to overlaps between the sets of the cover is resolved
with the help of a partition of unity $f_1,...,f_L$ subordinated to $\cU$: every point $z \in M$
is registered in exactly one of the subsets $U_i$ of the cover containing this point with probability
$f_i(z)$.

The {\it Berezin-Toeplitz quantization} (which exists for all closed symplectic manifolds
 whose symplectic form represents, up to a multiple $2\pi$, an integral cohomology class) is given by a seq\-uence of finite-dimensio\-nal complex Hilbert spaces $H_m$, $m \to \infty$ of increasing dimension and a family of  $\R$-linear maps $T_m: C^{\infty}(M) \to \cL(H_m)$ satisfying a number of axioms (most notably, the correspondence principle) which will be recalled later in Section \ref{sec-BT}. Here $\cL(H_m)$ stands for the space of Hermitian operators on $H_m$.
The number $\hbar = \frac{1}{m}$ represents the Planck constant, so that
$m \to \infty$ is the classical limit. With this language, the quantum version of the registration procedure is defined by means of the sequence of POVMs
\begin{equation}\label{eq-Am-intro}
A^{(m)}=\{T_m(f_i)\}
\end{equation}
on the finite space $\Omega_L=\{1,...,L\}$.
Being  prepared in a pure state
$[\xi] \in \mathbb{P} (H_m)$, $|\xi|=1$, the quantum system is registered in the set $U_i$ with probability
$\langle T_m(f_i) \xi,\xi\rangle$.

\subsection{Inherent noise}
The main character of our story is {\it the inherent noise} of POVMs coming from quantum registration procedures. This quantity, roughly speaking, measures the size of the non-random component of the quantum noise operator. The latter was studied earlier in \cite{BHL1,Ozawa,BHL2,Massar}. Let $B$ be an $\cL(H)$-valued POVM on a space $\Theta$. We say that a POVM
$A$ on $\Omega_L$ is a smearing (or randomization) of $B$ if there exists a measurable partition of unity
$\{\gamma_j\}$ on $\Theta$ such that $A_j = \int_\Theta \gamma_j dB$ for all $j$. This equality can be
interpreted as follows: every point $\theta \in \Theta$ diffuses into a point $j \in \Omega_L$
with probability $\gamma_j(\theta)$. To every random variable $x=(x_1,...,x_L)$ on $\Omega_L$
corresponds a random variable $\Gamma x = \sum_j x_j\gamma_j$ on $\Theta$
which has the same operator valued expectation as $x$ but whose operator valued variance
$$\Delta_B(\Gamma x):=  \int_\Theta (\Gamma x)^2\;dB - \Big{(}\int_\Theta \Gamma x \; dB\Big{)}^2$$
does not exceed the one of $x$ (see inequality \eqref{eq-MdeM} below).
{\it Is it possible to ``derandomize" $A$ in such a way that this
variance is small for all $x \in [-1,1]^L$?} The obstruction to such a derandomization  is given by the inherent noise of $A$ which is defined as
$$\cN_{in}(A) = \inf_{B} \max_{x \in [-1,1]^L}||\Delta_B(\Gamma x)||_{op}\;,$$
where the infimum is taken over all $B$ such that $A$ is a smearing of $B$. Here $||\;||_{op}$
stands for the operator norm.

A key feature of the inherent noise is as follows (see Section \ref{subsec-noise-indicator-single} below):
Let $A^{(m)}$ be the sequence of POVMs \eqref{eq-Am-intro}
associated with the quantum registration procedure.  Then there exists a constant $D>0$
such that for all $m \in \N$
\begin{equation}\label{eq-noise-upper-intro}
\cN_{in}(A^{(m)}) \leq D \cdot \hbar\;,
\end{equation}
where $\hbar = 1/m$.
The main finding of the present paper is that for certain classes of covers the inherent noise
 satisfies lower bounds of the form
$\cN_{in}(A^{(m)}) \geq C \cdot \hbar$ for all sufficiently large $m$, where the constant
 $C$ depends only on the symplectic geometry and combinatorics of the covers.
Our results can be briefly summarized as follows.

\subsection{Noise-localization uncertainty relation}
First, we discuss the principle stating that {\it `a sufficiently fine phase space localization yields the inherent quantum noise'}. Its qualitative version has been established in \cite{P}.
We present a quantitative version of this principle for the special class of open covers $\cU=\{U_1,...,U_L\}$ of $M$ satisfying assumptions (R1) and (R2) below. Let us write $\overline{U}$ for the closure of a subset $U \subset M$.

\medskip
\noindent {\bf R1.} Every subset $\overline{U}_j$ intersects closures of at most $d$ other subsets from the cover.

\medskip
\noindent Further, for a subset $X \subset M$
define its {\it star} $St(X)$ as the union of all $U_i$'s with $\overline{U}_i \cap \overline{X} \neq \emptyset$.

\medskip
\noindent {\bf R2.} For every $i$ there exists a time-dependent Hamiltonian function $F^{(i)}_t$ on $M$, $t \in [0;1]$ supported in the
$p$ times iterated star $St(...(St(U_i)...)$ of $U_i$ such that the time one map $\phi_i$ of the corresponding
Hamiltonian flow displaces $U_i$: $\phi_i(U_i) \cap \overline{U}_i =\emptyset$.

\begin{figure}[tbp]
\centering
\epsfig{file=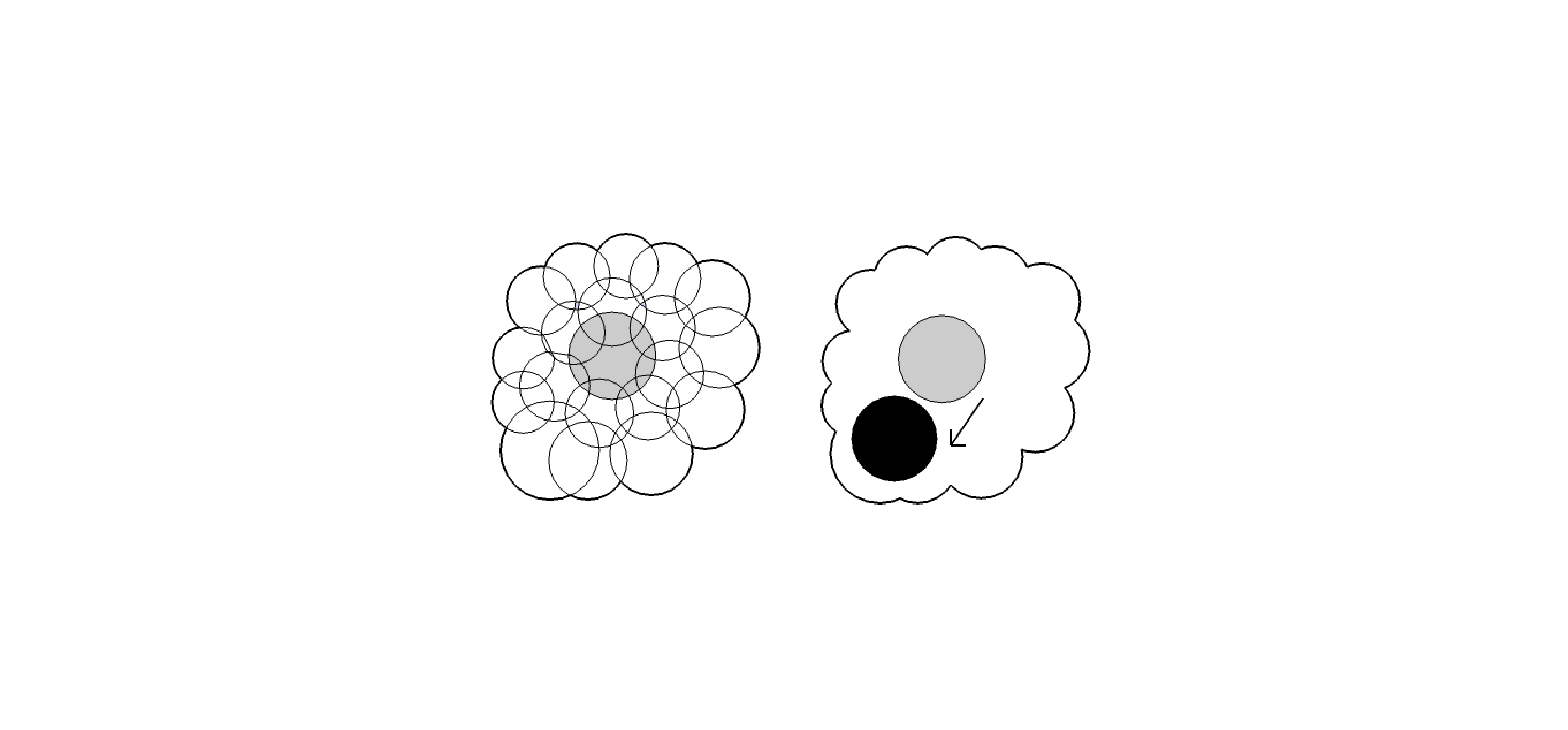,width=1.2\linewidth,clip=}
\caption{Displacement in two times iterated star}
\label{fig-star}
\end{figure}

\medskip
\noindent
Assumption (R2) is illustrated in Figure \ref{fig-star}. On the left one sees  $p=2$ times
iterated star of the gray disc $U_i$, and on the right -- its image (in black) under a Hamiltonian
diffeomorphism $\phi_i$ represented by an arrow.

\medskip
\noindent We call the covers satisfying (R1) and (R2) {\it $(d,p)$-regular}. We say that the {\it the magnitude of localization} of a $(d,p)$-regular cover $\cU$ is $\leq \cA$ if
$$\int_0^1 \max_M F^{(i)}_t - \min_M  F^{(i)}_t \; dt \leq \cA \;\;\; \forall i=1,...,L\;.$$
 In other words, we measure the size of subsets $U_i$ in terms of their {\it displacement energy} (cf. \cite{Hofer,Pbook}), that is of minimal amount of energy required in order to displace $U_i$ inside its $p$ times iterated star.
Let us mention that for certain $d,p$ depending only on $(M,\omega)$ one has the following: for every sufficiently small $\cA>0$ there exists a $(d,p)$-regular cover with magnitude of localization $\leq \cA$,
see Example \ref{exam-dpregular-greedy} below.

\medskip

Given a $(d,p)$-regular cover $\cU$, consider the sequence of POVMs $A^{(m)}$ given
by \eqref{eq-Am-intro}. Our first result is the following {\it noise-localization uncertainty relation:}

\begin{thm}\label{thm-intro-1}
\begin{equation}\label{noise-localization-single-intro}
\cN_{in}(A^{(m)}) \cdot \cA \geq C \hbar\;\;\; \forall m \geq m_0\;,
\end{equation}
where the positive constant $C$ depends only on $d$ and $p$.
\end{thm}

\medskip
\noindent
Let us note that number $m_0$ depends on the full data including the Berezin-Toeplitz quantization $T$ and the partition of unity $\{f_j\}$. We refer the reader to Section \ref{subsec-noise-local} below for the proof of
a more general version of this result and further discussion.

From a purely symplectic perspective, the noise-localization uncertainty relation is a manifestation
of rigidity of symplectic covers. Given a finite open cover $\cU=\{U_1,...,U_L\}$,
we introduce the associated Poisson bracket invariant
$$pb(\cU) = \inf_{\f} \max_{x,y\in [-1,1]^L}||\{\sum x_i f_i, \sum y_j f_j\}||\;,$$
where the infimum is taken over all partitions of unity $\f$ subordinated to
$\cU$ formed by smooth functions $f_1,...,f_L$.
Here  $\{.,.\}$ stands for the Poisson bracket and $||g||=\max_M |g|$.
This invariant serves as an obstruction to existence of a Poisson commutative partition
of unity subordinated to $\cU$. The study of lower bounds on $pb$ and related invariants in terms of geometry of covers was initiated in \cite{EPZ} and has been continued in the PhD-thesis of Frol Zapolsky as well as in \cite{P}. None of currently known bounds are sharp. For the applications to quantum mechanics we need the bounds which provide the sharp asymptotics in terms of the magnitude of localization of the cover.
Such bounds will be given in Section \ref{sec-mainresults} below.
Their proof involves Floer-theoretical methods of function theory on symplectic manifolds.
For instance, Theorem \ref{thm-intro-1} will follow from the estimate $pb(\cU) \geq C(p,d)\cA^{-1}$
for every $(d,p)$-regular cover.

Let us mention also that, in a different context, entropic uncertainty of quantized partitions of unity associated to covers of the cotangent bundles appeared in the seminal work by Anantharaman and Nonnenmacher \cite{AN}.

\subsection{Overlap induced noise}
Second, we present a different mechanism of the inherent quantum noise production based on geometry
of overlaps. Given a finite open cover $\cU= \{U_1,...,U_L\}$ of $M$, consider any decomposition of the set $\{1,...,L\}$ into disjoint union $I \sqcup I^c$ (upper index $c$ stands for the complement) and define {\it an overlap layer} by $\Lambda(\cU,I):= \bigcup_{\alpha \in I, \beta \in I^c} (U_\alpha \cap U_\beta)$.
Observe that $M \setminus \Lambda(\cU,I) = U(I) \sqcup U(I^c)$ where
$$U(I):=\Big{(}\bigcup_{\alpha \in I} U_\alpha\;\Big{)^c}\;.$$
We illustrate
these notions on Fig. \ref{fig-overlap-layer}: Here the cover consists of discs. The discs $U_\alpha$, $\alpha \in I$ form the
left column, the discs $U_\beta$, $\beta \in I^c$ form the right column, and the corresponding overlap layer
is colored in black. The sets $U(I)$ and $U(I^c)$ appear in gray (on the right) and in white (on the left)
respectively.

\begin{figure}[tbp]
\centering
\epsfig{file=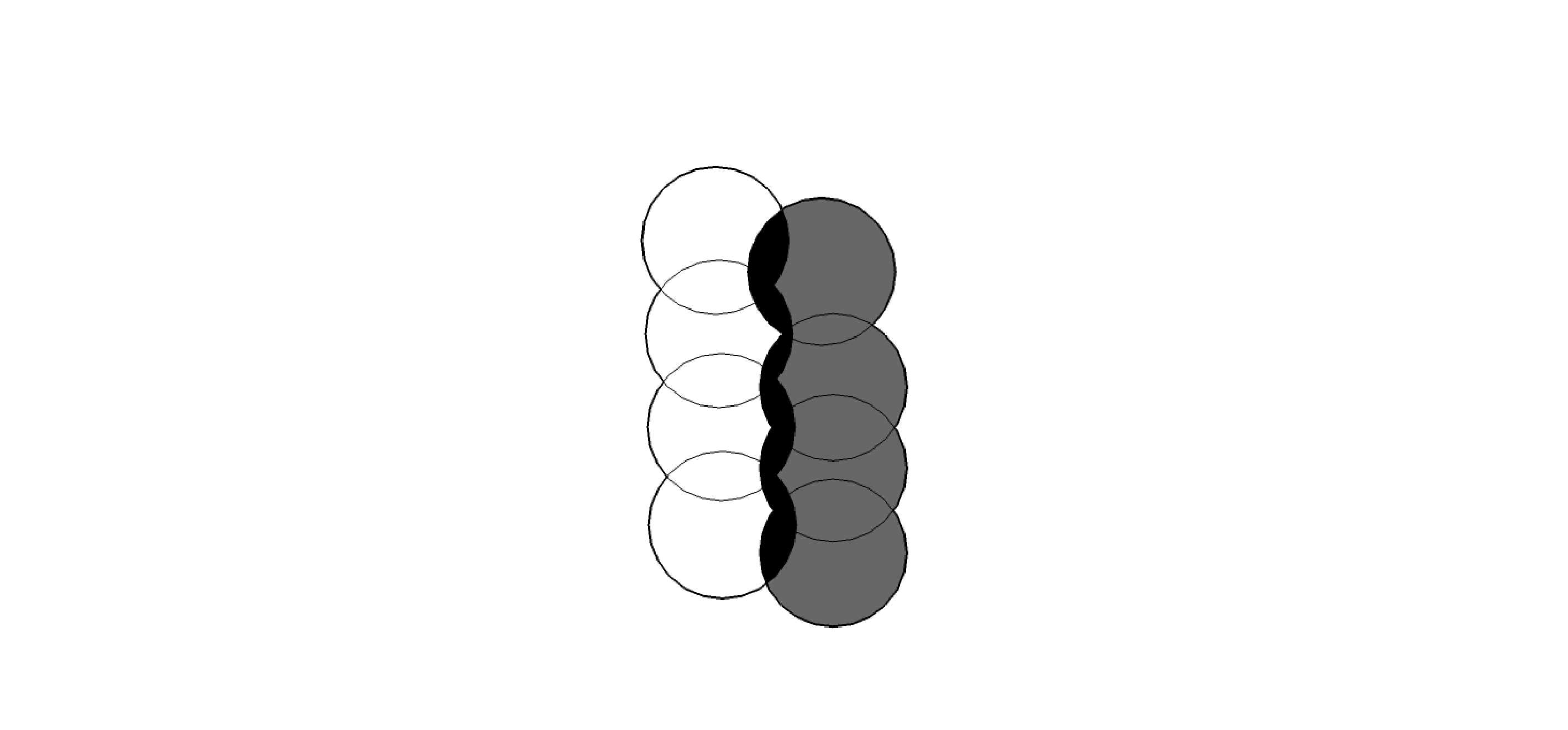,width=0.9\linewidth,clip=}
\caption{An overlap layer}
\label{fig-overlap-layer}
\end{figure}

Interestingly enough, the inherent quantum noise of a quantum registration procedure can be estimated in terms of quite subtle relative symplectic geometry of overlap layers of the corresponding cover . The main tool is the Poisson bracket invariant $pb_4$ introduced and studied in a recent paper \cite{BEP} by means of Floer theory: Let $X_0,X_1,Y_0,Y_1$ be a quadruple of compact subsets of a symplectic manifold $(M,\omega)$ satisfying
an intersection condition
$$X_0 \cap X_1 = Y_0 \cap Y_1 =\emptyset\;.$$
One of the equivalent definitions of $pb_4$ is as follows, see
\cite[Proposition 1.3]{BEP}:
$$pb_4(X_0,X_1,Y_0,Y_1)= \inf ||\{f,g\}||\;,
$$
where the infimum is taken over all pairs of functions $f,g: M \to [0;1]$ with
$f=0$ near $X_0$, $f=1$ near $X_1$, $g=0$ near $Y_0$ and $g=1$ near $Y_1$.

\medskip
\noindent
Given an open cover $\cU$ of $M$, consider the sequence of POVMs $A^{(m)}$ given by \eqref{eq-Am-intro}.

\medskip
\noindent
\begin{thm}\label{thm-intro-2} For every $I,J \subset \Omega_L$
\begin{equation}\label{eq-mu-pb4-overlap-intro}
\cN_{in}(A^{(m)}) \geq 2pb_4(U(I),U(I^c),U(J),U(J^c)) \cdot \hbar \;
\end{equation}
for all sufficiently large $m$.
\end{thm}

\medskip
\noindent
In Section \ref{subsec-overlap-concusion} below we shall prove a more general version of this result in the context of joint measurements of POVMs corresponding to a pair of quantum registration
procedures. It turns out that the POVMs associated to any pair of registration procedures admit a joint
observable (see Proposition \ref{prop-joint-BT} below). We shall see that in certain examples geometry of overlaps enables us to detect inherent quantum noise of such a joint observable even in the absence of phase space localization.

As a case study, we explore overlap-induced noise for a pair of special covers of the unit sphere $S^2$ by annuli. We argue that a joint observable for POVMs associated to these covers is responsible for an approximate joint measurement of two components of spin. By using the technique of overlap layers, we get a lower bound
for the inherent noise of such a measurement. Interestingly enough, it is related to the universal uncertainty relation for error bar widths of approximate joint measurements established earlier in \cite{Busch-3, Miyadera-2}. We refer to Sections \ref{sec-appendix} and \ref{sec-discussion}
for further details.

\medskip
\noindent {\sc Comparison with \cite{P}:} The study of quantum noise of POVMs associated to a quantum registration procedure via ``hard" symplectic constraints on the Poisson brackets
has been initiated in our recent paper \cite{P}. For reader's convenience, let us list the main new contributions of the present paper as compared with \cite{P}.
\begin{itemize}
\item We modify the measurement of non-random component of quantum noise by introducing inherent noise, which, in the context of quantum registration procedures,
satisfies $C\hbar \leq \cN_{in} \leq D\hbar$ in the classical limit $\hbar \to 0$  for some non-negative constants $C,D$. Thus the study of the value of the constant
$C$ (as opposed to the mere fact that $C>0$ which was proved in \cite{P}) becomes meaningful. This study eventually leads us to the noise-localization uncertainty relation \eqref{noise-localization-single-intro} which is a manifestation of the qualitative principle ``localization yields noise" established in \cite{P}. \item We find a new (as compared with \cite{P}) mechanism of quantum noise production based on
geometry of overlaps (see Theorem  \ref{thm-intro-2} above). In a somewhat unexpected twist, this
mechanism turns out to be related to the Poisson bracket invariant $pb_4$ introduced earlier
in \cite{BEP} and studied there by means of various flavors of theory of pseudo-holomorphic curves
in symplectic manifolds.
\item Following a suggestion by Paul Busch \cite{Busch-private} we extend our results to joint quantum measurements. In particular, {\it approximate} joint measurements (see Section \ref{sec-backto-quantum}  below) turn out to fit well the framework of quantum registration procedures.
    \end{itemize}
    Saying that, let us mention that we put an effort to make the present paper self-contained and
     hence there are inevitable overlaps with \cite{P}
    as far as various preliminaries from operational quantum physics, the Berezin-Toeplitz quantization
    and symplectic geometry are concerned.

    \medskip
\noindent {\sc Organization of the paper:}
In Section \ref{sec-Preliminaries on POVMs} we remind preliminaries
on POVMs and introduce the inherent quantum noise. The section is concluded with an unsharpness
principle for POVMs and their joint observables which provides a lower bound on the inherent noise
of a POVM in terms of its magnitude of non-commutativity.

In Section \ref{sec-registration}, after a brief reminder
on the Berezin-Toeplitz quantization, we describe the classical registration procedure and its quantum counterpart. Next we define  {\it the noise indicator} which enables us to discuss in a concise language inherent quantum noise of POVMs associated to quantum registration procedures. Finally, we introduce the Poisson bracket invariants of finite open covers of symplectic manifolds which provide lower bounds on the inherent quantum noise.

In Section \ref{sec-mainresults} we study Poisson bracket invariants of fine open covers of symplectic manifolds. The highlight of this section is Theorem \ref{thm-main-fine covers} providing lower
bounds on the Poisson bracket invariants in terms of geometric and combinatorial properties
of covers. The proofs involving methods of function theory on symplectic manifolds are presented in Section \ref{sec-symptop}.

In Section \ref{subsec-pb4-overlaps} we detect the overlap induced noise by using $pb_4$-Poisson
bracket invariant introduced in \cite{BEP} (see Theorem \ref{thm-overlap-layer-pb4}).

In Section \ref{sec-backto-quantum} we present various applications of our results
to quantum mechanics: First, we state and prove the noise-localization uncertainty relation for special classes of fine open covers. Next, after discussing a  link between the registration procedure and approximate quantum measurements, we detect the overlap induced noise for approximate measurements of two components of spin.

The paper is concluded with some open problems stated in Section \ref{sec-further-directions}.

\section{Operational quantum mechanics}\label{sec-Preliminaries on POVMs}

\subsection{Preliminaries on POVMs}
Let $H$ be a complex Hilbert space. In the present paper we deal with finite-dimensional
spaces only.  Denote by $\cL(H)$
the space of all Hermitian bounded operators on $H$. Consider a set $\Omega$ equipped with a $\sigma$-algebra
$\cC$ of its subsets. An $\cL(H)$-valued {\it positive operator valued measure} $A$ on $(\Omega,\cC)$ is
a countably additive map $A: \cC \to \cL(H)$ which takes a subset $X \in \cC$ to a positive operator
$A(X) \in \cL(H)$ and which is normalized by $A(\Omega) = \id$.

POVMs naturally appear in quantum measurement theory \cite{Busch} where they play a role
of generalized observables. The space $\Omega$ is called
the {\it value space} of the observable. Pure states of the system are represented by the points
of the projective space $[\xi] \in \mathbb{P} (H)$, where $\xi \in H$ is a unit vector. When the system is in a state $[\xi]$, the probability of finding the observable $A$ in a subset $X \in \cC$ is postulated to be
$\langle A(X)\xi,\xi \rangle$.

An important class of POVMs is formed by {\it projector valued measures} for which all the operators
$A(X)$, $X \in \cC$ are orthogonal projectors. In the language of quantum measurement theory they are called
{\it sharp} observables. Every  ``usual" von Neumann observable $B \in \cL(H)$ corresponds to the projector
valued measure $\widehat{B}= \sum P_j\delta_{\lambda_j}$ on $\R$, where $B=\sum_j \lambda_j P_j$ is the spectral decomposition of $B$. In this case the statistical postulate above agrees with the one of von Neumann's quantum mechanics.

When $\Omega=\Omega_L:=\{1,...,L\}$ is a finite set, any POVM $A$ on $\Omega$ is fully determined by
$L$ positive Hermitian operators $A_i:= A(\{i\})$ which sum up to $\id$.

\subsection{Smearing of POVMs}\label{subsec-smearing}  Let $\Omega$ and $\Theta$ be Hausdorff locally compact second countable topological spaces and let $\cC,\cD$ be their Borel $\sigma$-algebras respectively. Denote by $\cP(\Omega)$ the set of Borel probability measures on $\Omega$. A {\it Markov kernel} is a map
$$\gamma: \Theta \to \cP(\Omega),\; w \mapsto \gamma_w$$
such that the function $w \to \gamma_w(X)$ on $\Theta$ is measurable for every $X \in \cC$.
Let $A$ and $B$ be POVMs on $(\Omega,\cC)$ and $(\Theta, \cD)$ respectively. We say that
$A$ is a {\it smearing}\footnote{Some authors call it randomization or fuzzification.} of $B$
\cite{Busch, JP, Ali} if there exists a Markov kernel $\gamma$ so that
$$A(X) = \int_\Theta \gamma_w(X)\; dB(w)\;\; \forall X \in \cC\;.$$
In the physical language, each element $w$ of the value set $\Theta$ of $B$ diffuses into
a subset $X \in \cC$ with probability $\gamma_w(X)$.

\medskip

Denote by $K(\Omega)$ the set of all measurable functions $x:\Omega \to \R$ with $\max |x| \leq 1$
which are interpreted as random variables on $\Omega$.
Every Markov kernel $\gamma$ as above defines a {\it smearing operator}
\begin{equation}\label{eq-smearing operator}
\Gamma: K(\Omega) \to K(\Theta),\; (\Gamma x)(w) = \int_{\Omega} x \; d\gamma_w\;\;\forall w \in \Theta\;.
\end{equation}
In the physical slang, the random variable $\Gamma x$ is a coarse-graining of $x$: Its value at
a point $w \in \Theta$ is the expectation of $x$ with respect to the measure $\gamma_w$.

For instance, if $\Omega= \Omega_N=\{1,...,N\}$,
a Markov kernel $\gamma$ is given by a collection of non-negative measurable functions
$\gamma_j$ on $\Theta$, $j=1,...,N$ so that $\sum_j \gamma_j(w)=1$ for all $w \in \Theta$,
that is by a measurable partition of unity.
A POVM $A=\{A_1,...,A_N\}$ on $\Omega_N$,
\begin{equation} \label{eq-smearing-int}
A_j = \int_\Theta \gamma_j dB\;,
\end{equation}
is a smearing of a POVM $B$ on $(\Theta,\cD)$.
Every function $x$ from $K(\Omega_N)$ is canonically identified with
a vector $x=(x(1),...,x(N))$ lying in the cube $K_N := [-1,1]^N$.
The smearing operator $\Gamma$ is given by
$\Gamma(x) = \sum_j x_j\gamma_j$.

\subsection{Quantum noise}
Let $A: \cC \to \cL(H)$ be an $\cL(H)$-valued POVM on $(\Omega,\cC)$.
For a function $x \in K(\Omega)$ put $A(x):= \int x\; dA$ and define the noise operator $$\Delta_A(x) = \int_\Omega x^2 dA - A(x)^2$$ (see \cite[\S 2]{BHL1}, \cite[\S 4]{Ozawa}, \cite[\S 3]{BHL2}, \cite[\S 2]{Massar}).
Roughly speaking, $A(x)$ is the operator valued expectation of the random variable $x$ with respect to POVM $A$, while $\Delta_A(x)$ is its operator-valued variance.

Let us give a more precise interpretation of the noise operator
(see e.g. \cite{Massar}). Fix a quantum state $\xi \in H$, $|\xi|=1$. Introduce the following
pair of random variables, $\phi$ and $\psi$: Consider a measure $\sigma_\xi$ on $\Omega$ given by $\sigma_\xi(X) = \langle A(X)\xi,\xi\rangle$. We define $\phi$ as $x: \Omega \to \R$,
where $\Omega$ is equipped with the probability measure $\sigma_\xi$. The random variable $\psi$
is associated to the von Neumann observable $A(x)$ and the state $\xi$:
It takes values $\lambda_j$ with probability $\langle P_j\xi,\xi \rangle$, where $A(x) =\sum \lambda_jP_j$
is the spectral decomposition.  Both random variables have the same expectation $\langle A(x)\xi,\xi \rangle$, while the difference of their variances
is given in terms of the noise operator:
$$\langle \Delta_A(x)\xi,\xi \rangle= \text{Var}(\phi) - \text{Var}(\psi)\;.$$
We refer to \cite{Massar,P} for the proof and a more detailed discussion.

The noise operator has a number of interesting properties, in particular $\Delta_A(x) \geq 0$.
The equality $\Delta_A(x)=0$ for all $x$ corresponds precisely to the case when $A$ is a projector valued
measure.

\medskip

Introduce the {\it magnitude of noise} $\cN(A):= \max_{K(\Omega)} ||\Delta_A(x)||_{op}$,
where $K(\Omega)$ is defined in the previous section and $||\;.||_{op}$ stands for the operator norm.
The magnitude of noise is a dimensionless quantity satisfying (see \cite{P})
\begin{equation}\label{eq-01}
0 \leq \cN(A) \leq 1\;.
\end{equation}

\medskip
\noindent
Now we are ready to introduce one of the central notions of the present paper.

\medskip
\noindent
\begin{defin}\label{defin-steady noise}{\rm The {\it inherent noise} of a POVM $A$ on $(\Omega,\cC)$
is given by
$$\cN_{in}(A):= \inf_{B,\Gamma} \sup_{x \in K(\Omega)} ||\Delta_B(\Gamma x)||_{op}\;,$$
where the infimum is taken over all POVMs $B$ so that $A$ is a smearing of $B$ and $\Gamma$ is the corresponding smearing operator.}
\end{defin}

\medskip
\noindent
\begin{prop}\label{prop-steady-usual}
$\cN_{in}(A) \leq \cN(A)$.
\end{prop}
\begin{proof} By a result of Martens and de Muynck \cite[Lemma 2, p.277]{MdeM}
\begin{equation}\label{eq-MdeM} \Delta_B(\Gamma x) \leq \Delta_A(x)\;,
\end{equation}
where $B$ is any smearing of $A$ with the smearing operator $\Gamma$ and $x$ is any function
from $K(\Omega)$. This immediately yields the proposition.

For the sake of completeness, let us prove inequality \eqref{eq-MdeM}.
Observe that $B(\Gamma x) = A(x)$. Furthermore, in the notation of Section \ref{subsec-smearing},
one calculates that
$$\int_\Omega x^2 \; dA -  \int_\Theta (\Gamma x)^2 \; dB= \int_\Theta h(w) dB(w)\;,$$
with
$$h(w) = \int_\Omega x^2 \; d\gamma_w - \Big{(}\int_\Omega x \; d\gamma_w \Big{)}^2\;.$$
Since $h(w) \geq 0$ by the Cauchy-Schwarz inequality, we get \eqref{eq-MdeM}.
\end{proof}

\medskip
\noindent Since smearing can be interpreted as a diffusion, the increment $\Delta_A(x) - \Delta_B(\Gamma x)$ (which is always non-negative by \eqref{eq-MdeM}) plays the role of
a random  component of the noise of $A$. That's why we call $\cN_{in}(A)$ the {\it inherent} (as opposed to random) noise. Let us mention that in \cite{P} a non-random component of the noise of $A$ was measured in a slightly different way:
 we defined {\it a systematic} noise of a POVM $A$ as $\inf_B \cN(B)$ where the infimum is taken over all $B$ so that $A$ is a smearing of $B$. By definition, $||\Delta_B(\Gamma x)||_{op} \leq \cN(B)$, and hence $\cN_{in}(A) \leq \cN_s(A)$. The inherent noise has a significant advantage in comparison with the systematic noise: In the context of quantum registration procedures, one can find an upper bound on $\cN_{in}$
having the correct asymptotic behavior with respect to the quantum number (see Section \ref{subsec-noise-indicator-single} below), while the similar
problem for $\cN_s$ is at the moment out of reach. At the same time, both $\cN_{in}$ and $\cN_s$ obey
an unsharpness principle which we are going to discuss in the next section.

\medskip

It turns out that the inherent noise behaves monotonically under smearings.

\medskip
\noindent
\begin{prop} \label{prop-behaviour-smearings}
$\cN_{in}(A) \leq \cN_{in}(B)$ provided $A$ is a smearing of $B$.
\end{prop}

\begin{proof}
Let us introduce the following notation: For a POVM $A$ we write $\Omega^A$ for the value space
of $A$ and we abbreviate $K^A := K(\Omega^A)$. Following \cite{MdeM}, we write $B \to A$ if $A$ is a smearing of $B$, and denote by $\Gamma_{BA}: K^A \to K^B$ the smearing operator.

Observe that if $C \to B \to A$ we have that $\Gamma_{CA}=\Gamma_{CB}\Gamma_{BA}$. Therefore
$$\sup_{x \in K^A} ||\Delta_C(\Gamma_{CA} x)||_{op} = \sup_{x \in K^A} ||\Delta_C(\Gamma_{CB}\Gamma_{BA} x)||_{op} \leq \sup_{y \in K^B} ||\Delta_C(\Gamma_{CB}y)||_{op}\;,$$ and hence
$$\inf_{C:\; C \to B \to A} \sup_{x \in K^A} ||\Delta_C(\Gamma_{CA} x)||_{op} \leq \inf_{C:\; C \to B} \sup_{y \in K^B} ||\Delta_C(\Gamma_{CB}y)||_{op} = \cN_{in}(B)\;.$$
Note that
$$\{C\;:\; C \to B \to A\} \subset \{C: C \to A\}\;.$$
Thus the left hand side of the last inequality is $\geq \cN_{in}(A)$.
We conclude that $\cN_{in}(A) \leq \cN_{in}(B)$, as required.
\end{proof}

\subsection{An unsharpness principle}
For an $\cL(H)$-valued POVM $A$ on $(\Omega,\cC)$ define the magnitude of non-commutativity
\begin{equation}\label{eq-noncom}
\nu_q(A) = \sup_{x,y \in K(\Omega)} ||[A(x),A(y)]||_{op}\;,
\end{equation}
where $[.,.]$ denotes the commutator and the subindex $q$ stands for quantum.

\medskip
\noindent
\begin{thm} [Unsharpness principle] \label{thm-unsharpness}
\begin{equation}\label{eq-unshprsingle}
\cN_{in}(A) \geq \frac{1}{2} \nu_q(A)\;.
\end{equation}
\end{thm}

\begin{proof} Let $B$ be an $\cL(H)$-valued POVM on $(\Theta,\cD)$.
We shall use the following inequality which appears in
\cite{I,Janssens,MI,P} as well in \cite[Theorem 7.5]{Hayashi}:
\begin{equation}\label{eq-janssens-0}
||\Delta_B(u)||^{1/2}_{op}\cdot ||\Delta_B(v)||^{1/2}_{op} \geq \frac{1}{2} \cdot ||\;[B(u),B(v)]\;||_{op}\;
\end{equation}
for all $u,v \in K(\Theta)$.

Suppose now that an $\cL(H)$-valued POVM $A$ on $(\Omega,\cC)$ is a smearing
of $B$ with the smearing operator $\Gamma$. Take any $x,y \in K(\Omega)$ and substitue
$u = \Gamma x$ and $v = \Gamma y$ into the above inequality. Since $B(u) = A(x)$ and $B(v)=A(y)$
we get that
\begin{equation} \label{eq-janssens}
||\Delta_B(\Gamma x)||^{1/2}_{op}\cdot ||\Delta_B(\Gamma y)||^{1/2}_{op} \geq \frac{1}{2} \cdot ||\;[A(x),A(y)]\;||_{op}\;.
\end{equation}
Therefore
$$\sup_x ||\Delta_B(\Gamma x)||_{op} \geq \frac{1}{2} \nu_q(A)\;,$$
which yields the unsharpness principle.
\end{proof}

It is known \cite[Section 5]{Ali} (cf. \cite{JP}) that every commutative POVM on a Hausdorff locally compact second countable space is necessarily a smearing of a sharp observable, that is of a projector valued measure. In particular, $\cN_{in}(A)=0$ provided $\nu_q(A)=0$. The unsharpness principle above shows that the converse statement is also true.

Let us mention also that the magnitude of non-commutativity behaves monotonically with respect to smearings
(see \cite{P}):
\begin{equation}\label{eq-monoton}
\nu_q(B) \geq \nu_q(A)\;\; \text{if}\;\; A\;\;\text{is a smearing of}\;\; B\;.
\end{equation}
Thus after a smearing both sides of inequality \eqref{eq-unshprsingle}
decrease.

\subsection{An unsharpness principle for joint measurements}\label{sec-unsharpness}
In this section we deal with POVMs defined on finite sets of the form $\Omega_L=\{1,...,L\}$.
Recall that two $\cL(H)$-valued POVMs $A$ and $B$ on $\Omega_L$ and $\Omega_N$ are {\it jointly measurable} \cite{Busch} if there exists a POVM, say $C=\{C_{ij}\}$ on $\Omega:=\Omega_L \times \Omega_N$
whose marginals equal $A$ and $B$:
$$A_i =\sum_j C_{ij},\;\; B_j = \sum_i C_{ij}\;\; \forall i \in \Omega_L,\;j\in \Omega_N\;.$$
Such a POVM $C$ is called {\it a joint observable} for $A$ and $B$. Let us emphasize
that the question on joint measurability is a delicate one, and not every two POVMs admit a joint
observable.

Given two POVMs $A,B$ on $\Omega_L$ and $\Omega_N$ respectively,
define the magnitude of (mutual) non-commutativity of $A$ and $B$ by
$$\nu_q(A,B) = \max_{x \in K_L,y\in K_N} ||[A(x),B(y)]||_{op} \;,$$
where as earlier  $A(x):=\sum x_i A_i$ and $K_L = [-1;1]^L$.
With this notation, $\nu_q(A,A) = \nu_q(A)$, where the latter quantity is introduced in \eqref{eq-noncom}
above.

\medskip
\noindent
\begin{prop}[Unsharpness principle for joint measurements]\label{prop-joint-unsharpness}
Let $C$ be a joint observable for $\cL(H)$-valued POVMs $A$ and $B$ on $\Omega_L$ and $\Omega_N$ respectively
Then
\begin{equation}
\label{eq-joint}
\cN_{in}(C) \geq \frac{1}{2}\cdot\max\Big{(}\nu_q(A),\nu_q(B), \nu_q(A,B)\Big{)}\;.
\end{equation}
\end{prop}

\medskip
\noindent
We refer to \cite{MI} for related results.

\begin{proof}
Indeed, take any $x \in K_L$, $y \in K_N$ and note that
both $A(x)$ and $B(y)$ are linear combinations of $C_{ij}$ with coefficients
lying in $[-1;1]$. Thus
$\nu_q(C) \geq \nu_q(A,B)$. Furthermore, $A$ and $B$ are smearings
of $C$, and thus by \eqref{eq-monoton} $\nu_q(C) \geq \nu_q(A)$ and
$\nu_q(C) \geq \nu_q(B)$. Therefore
$$\nu_q(C) \geq  \max\Big{(}\nu_q(A),\nu_q(B), \nu_q(A,B)\Big{)}\;,$$
and so \eqref{eq-joint} follows from the unsharpness principle \eqref{eq-unshprsingle}. \end{proof}

\section{Classical and quantum registration procedures}\label{sec-registration}

\subsection{The Berezin-Toeplitz quantization}\label{sec-BT}
Recall that by the classical Darboux theorem, near each point
of a symplectic manifold $(M,\omega)$ one can choose local coordinates $p_1,q_1,...,p_n,q_n$ so that in these coordinates $\omega = \sum_{j=1}^n dp_j \wedge dq_j$. The space $C^{\infty}(M)$ of smooth functions on $M$
is equipped with {\it the Poisson bracket} $\{f,g\}$, which in the Darboux coordinates $(p,q)$ is given by $$\{f,g\}(p,q) = \frac{\partial f}{\partial q} \cdot \frac{\partial g}{\partial p}-\frac{\partial f}{\partial p} \cdot \frac{\partial g}{\partial q}\;.$$

Let $(M,\omega)$ be a closed symplectic manifold such that the form $\omega/2\pi$ represents
an integral de Rham cohomology class. In what follows such symplectic manifolds will be called {\it quantizable}. The Berezin-Toeplitz quantization \cite{Berezin,BMS,Gu,BU,MM,S}
consists of a sequence of finite-dimensional complex Hilbert spaces $H_m$, $m \to \infty$ of the increasing dimension and a family
of surjective $\R$-linear maps $T_m: C^{\infty}(M) \to \cL(H_m)$ with the following properties:
\begin{itemize}
\item[{(BT1)}] $T_m(1) =\id$;
\item[{(BT2)}] $T_m(f) \geq 0$ for $f \geq 0$;
\item[{(BT3)}]  $||T_m(f)||_{op} = ||f||+ O(1/m)$;
\item[{(BT4)}] ({\it the correspondence principle}) $||-i \cdot m[T_m(f),T_m(g)] - T_m(\{f,g\})||_{op} = O(1/m)$;
\item[{(BT5)}]  $||T_m(f^2) - T_m(f)^2||_{op} = O(1/m)$,
\end{itemize}
as $m \to \infty$ for all $f,g \in C^{\infty}(M)$.
Here $||f|| = \max |f(x)|$ stands for the uniform norm of a function $f \in C^{\infty}(M)$,
$\{f,g\}$ for the Poisson bracket, $||A||_{op}$ for the operator norm of $A \in \cL(H)$ and
$[A,B]$ for the commutator $AB-BA$. The number $m$ plays the role of the quantum number,
while the Planck constant $\hbar$ equals $1/m$, so that $m \to \infty$ is the classical limit.
As an immediate consequence of (BT3) and (BT4) we get that
\begin{multline}\label{eq-pbquantbound}
||m[T_m(f),T_m(g)]\;||_{op} = ||T_m(\{f,g\})||_{op}+ O(1/m) \\ = ||\{f,g\}|| + O(1/m)
\end{multline}
for all smooth functions $f,g$ on $M$.

\medskip

The Berezin-Toeplitz quantization can be described in the language of POVMs : There exists a sequence of $\cL(H_m)$-valued POVMs $G_m$ on the symplectic manifold $M$ equipped with the Borel $\sigma$-algebra so that
\begin{equation}\label{eq-POVM_Toeplitz}
T_m(f)=\int_M f\; dG_m\;.
\end{equation}
This follows from  Proposition 1.4.8 of Chapter II in \cite{L} (the argument of \cite{L} is repeated in \cite{P}).
Property (BT5) deserves a special discussion. Let $f,g \in C^{\infty}(M)$ be any two functions.
Assume without loss of generality that $||f||\leq 1$ and $||g|| \leq 1$. Then
$\Delta_{G_m} (f) = T_m(f^2) - T_m(f)^2$ and hence by \eqref{eq-janssens-0} and \eqref{eq-pbquantbound} above
\begin{multline}\label{eq-surprising}
|| T_m(f^2) - T_m(f)^2||_{op}^{1/2} \cdot || T_m(g^2) - T_m(g)^2||_{op}^{1/2} \\ \geq \frac{1}{2}\cdot ||[T_m(f),T_m(g)]||_{op} = ||\{f,g\}||/(2m) + O(1/m^2)\;.
\end{multline}
Assume now that the function $f$ is non-constant. Then there exists a function $g$ with $\{f,g\} \neq 0$,
and we conclude that   $T_m(f^2) \neq T_m(f)^2$ for all sufficiently large $m$.

\subsection{Registration}\label{subsec-registration}
 Let $(M,\omega)$ be a closed connected symplectic manifold playing the role of the phase space of a classical system. Take a finite open cover  $\cU=\{U_1,...,U_L\}$ of $M$. A partition of unity $\f=\{f_1,...,f_L\}$ subordinated to $\cU$ is a collection of smooth non-negative
real-valued functions $f_i$ on $M$ with $\text{supp}(f_i) \subset U_i$ and $f_1+...+f_L=1$.
Here $\text{supp}(f)$ stands for the support of the function $f$, that is for the closure of the set
$\{f\neq 0\}$. Such partitions of unity naturally appear in the following {\it registration procedure}
\cite{P}:  every point $z \in M$ is registered with probability $f_i(z)$ in exactly one of the subsets $U_i$ of the cover containing this point. This procedure can be considered as an attempt of phase space
localization of the classical system with respect to the cover $\cU$: It yields an answer to the question {\it `Where (i.e. in which set $U_j$) is the system located?'}

\medskip
\noindent \begin{exam}\label{exam-greedy} {\rm Fix a Riemannian metric on $M$ and a number $r >0$
which is sufficiently small in comparison with the injectivity radius of $M$. Let $\rho$ be the correspondent
distance function. For a point $z \in M$ denote by $D(z,r)$ the open metric ball of radius $r$ centered at $z$. A (necessarily, finite) collection of distinct points $\{z_1,...,z_L\}$ in $M$ is called {\it $r/2$-separated} (cf. \cite{BBB-metric-geometry}) if $\rho(z_i,z_j) \geq r/2$ for all $i\neq j$. Take any $r/2$-separated collection and add points to it keeping it $r/2$-separated until the process terminates. The resulting collection, say, $Z$ is {\it maximal}: there is no $r/2$-separated collection containing it as a proper subset. In other words, every point $z \in M$ lies at distance $< r/2$ from some $z_j \in Z$. Therefore the balls $D(z,r/2)$, $z \in Z$ cover $M$. We call the cover $\{D(z,r)\}$, $z \in Z$ {\it greedy} (by technical reasons,
it is convenient to deal with balls of radius $r$ as opposed to $r/2$). The corresponding registration problem reflects a naive attempt of phase space discretization.
}\end{exam}

\medskip
\noindent \begin{exam}\label{exam-one-func}{\rm  Take a smooth function $F: M \to \R$. Consider the closed interval $I:= [\min F,\max F]$. Let $\{W_i\}$ be a cover of $I$ by open intervals of length $< c$, and let $\{h_i\}$ be smooth
functions  $\R \to [0,1]$ so that $\sum h_i =1$ on $I$ and each $h_i$ is supported in $W_i$. Then the functions $h_i\circ F$ form a partition of unity subordinated to the open cover $\{F^{-1}(W_i)\}$
of $M$. The outcome of the corresponding registration procedure can be spelled out as follows:
the value of $F$ at a given point $z \in M$ lies in the interval $W_i$ with the probability
$p_i = h_i(F(z))$. Thus the registration procedure yields an approximation to the genuine value
$F(z)$ with the error $c$. This example will serve as a starting point of our discussion
on approximate measurements in Section \ref{sec-appendix} below.
}\end{exam}

In order to describe a quantum counterpart of the registration procedure, fix a scheme $T_m$ of the
Berezin-Toeplitz quantization. It takes the partition of unity $\{f_i\}$ to an $\cL(H_m)$-valued POVM $A^{(m)}$ on the finite set $\Omega_L=\{1,...,L\}$, where $A^{(m)}_i:= T_m(f_i)$. Being  prepared in a pure state
$[\xi] \in \mathbb{P} (H_m)$, $|\xi|=1$, the quantum system is registered in the set $U_i$ with probability
$\langle T_m(f_i) \xi,\xi\rangle$.

\subsection{The noise indicator: a single partition of unity}\label{subsec-noise-indicator-single}
Fix a scheme of the Berezin-Toeplitz quantization
$T_m: C^{\infty}(M) \to \cL(H_m)$. For a partition of unity $\f= \{f_i\}$, $i=1,...,L$  of $M$
consider the POVM $A^{(m)} = \{T_m(f_j\}$ and focus on the sequence of non-negative numbers
$m \cdot \cN_{in}(\{T_m(f_i)\}$, $m \geq 0$, where $\cN_{in}$ stands for the inherent noise.
First of all we claim that this sequence is necessarily bounded from above (this justifies
inequality \eqref{eq-noise-upper-intro} from the introduction).
Indeed, observe that
$A^{(m)}$ is a smearing of the Berezin-Toeplitz POVM $G_m$, where the smearing operator is given by
the partition of unity:
$$\Gamma x =  \sum_i x_i f_i\;$$ for every
$x \in K_L=[-1,1]^L$.
The noise operator can be written as
$$\Delta_{G_m}(\Gamma x) = T_m( (\Gamma x)^2) - T_m(\Gamma x)^2\;.$$
Since the functions $\Gamma x$, $x \in K_L=[-1,1]^L$ form a compact subset in $C^\infty$-topology,
the property (BT5) of the Berezin-Toeplitz quantization yields the bound
$$\sup_{x \in K_L} ||\Delta_{G_m}(\Gamma x)||_{op} \leq D/m $$
for some $D > 0$. By definition, the left hand side of this inequality is $\geq \cN_{in}(A^{(m)})$,
and hence $m \cdot \cN_{in}(\{T_m(f_i)\}) \leq D$. The claim follows.
Therefore, the quantity
$\mu_T(\f):= \liminf_{m \to \infty} m \cdot \cN_{in}(\{T_m(f_i)\})$
is necessarily finite. Now we are ready to define one of the central notions of the present paper.

\medskip
\noindent
\begin{defin} \label{def-noise-indicator-single} {\rm Let $\cU=\{U_1,...,U_L\}$ be
a finite open cover of a closed quantizable symplectic manifold $(M,\omega)$.
The {\it noise indicator} of $\cU$ is defined as
$$\mu(\cU):= \inf_{\f,T} \mu_T(\f)\;,$$
where the infimum is taken over all partitions of unity $\f$ subordinated to $\cU$
and over all schemes $T_m$ of the Berezin-Toeplitz quantization.}
\end{defin}

\subsection{The noise indicator: joint measurements} \label{subsec-noise-indicator-joint}
Consider POVMs $A^{(m)}=\{T_m(f_i)\}$ and $B^{(m)}= \{T_m(g_j)\}$,
where $\f= \{f_i\}$ and $\g= \{g_j\}$ are partitions of unity on $M$.

\medskip
\noindent
\begin{prop}\label{prop-joint-BT}
POVMs $A^{(m)}$ and $B^{(m)}$ necessarily admit a joint observable
on $\Omega:= \Omega_L \times \Omega_N$.
\end{prop}

\medskip
\noindent Let us emphasize that a joint observable is not unique
and does not necessarily has a classical counterpart.

\medskip
\noindent
\begin{proof}
Put $D^{(m)}_{ij}:= T_m(f_ig_j)$.
Note that since $f_ig_j \geq 0$ these operators are positive. Since $\sum f_ig_j=1$
and $T_m$ is linear with $T_m(1)=\id$, these operators sum up to $\id$. Thus $\{D^{(m)}_{ij}\}$
form a POVM, say $D^{(m)}$ on $\Omega$. Let us check that its marginals are precisely $A^{(m)}$ and $B^{(m)}$.
Indeed, since $\sum_i f_ig_j = g_j$ we get (again by linearity of $T_m$) that
$$\sum_i D^{(m)}_{ij}= T_m(g_j) = B^{(m)}_j\;,$$
and similarly for $A^{(m)}$.
\end{proof}

\medskip

Let
$$\mu_T (\f,\g) = \inf_{C^{(m)}} \liminf_{m \to \infty} m \cdot \cN_{in}( C^{(m)})\;,$$
where the infimum is taken over all joint observables $C^{(m)}$ of $A^{(m)}$ and $B^{(m)}$.
Put $\f \cdot \g := \{f_ig_j\}$. Since $\{T_m(f_ig_j)\}$ provides such a joint observable,
$\mu_T(\f,\g) \leq \mu_T(\f \cdot \g) < \infty$.

\medskip
\noindent
\begin{defin} \label{def-noise-indicator-joint} {\rm Let $\cU=\{U_1,...,U_L\}$
and $\cV =\{V_1,...,V_N\}$ be a pair of finite open covers of a closed quantizable symplectic manifold $(M,\omega)$. The {\it noise indicator} of the pair $\cU,\cV$ is defined as
$$\mu(\cU,\cV):= \inf_{\f,\g, T} \mu_T(\f,\g)\;,$$
where the infimum is taken over all partitions of unity $\f$ and $\g$ subordinated to $\cU$ and $\cV$
respectively and over all schemes $T_m$ of the Berezin-Toeplitz quantization.}
\end{defin}

\medskip
\noindent
Observe that the POVM $A_{ij} = A_i \cdot \delta_{ij}$, where $\delta_{ij}$ is the Kronecker delta,
is a joint observable for two copies of a POVM $A=\{A_i\}$. Therefore
\begin{equation}\label{eq-two-noises}
\mu(\cU,\cU) \leq \mu(\cU)\;.
\end{equation}

\subsection{The Poisson bracket invariants}
The objective of the present paper is to find lower bounds on the noise indicators
in terms of symplectic geometry and topology of the covers. The first step in this direction
is provided by the correspondence principle (BT4).

For a partition of unity $\f=\{f_i\}$, $i=1,...,L$
introduce the magnitude of its Poisson non-commutativity
$$\nu_c(\f):=  \max_{x \in K_L, y\in K_L} ||\{\sum_i x_i f_i,\sum_i y_if_i\}||\;,$$
where as above $K_L$ stands for the cube $[-1;1]^L \subset \R^L$ (cf. \eqref{eq-noncom} above; the subindex
$c$ stands for classical). Given an open cover
$\cU=\{U_1,...,U_L\}$ of $M$, define the Poisson bracket invariant
$pb(\cU) = \inf \nu_c(\f)$, where the infimum is taken over all partitions of unity $\f$
subordinated to $\cU$.

Similarly, given a pair of partitions of unity $\f= \{f_i\}$, $i=1,...,L$ and $\g=\{g_j\}$,
$j=1,...,N$, put
$$\nu_c(\f,\g):=  \max_{x \in K_L, y\in K_N} ||\{\sum_i x_i f_i,\sum_j y_j g_j\}||\;.$$
Introduce the Poisson bracket invariant of two open covers $\cU$ and $\cV$:
\begin{equation}
\label{eq-pb}
pb(\cU,\cV) = \inf_{\f,\g}\max\Big{(}\nu_c(\f),\nu_c(\g),\nu_c(\f,\g)\Big{)} \;,
\end{equation}
where the infimum is taken over all partitions of unity subordinated to our covers.

\noindent
Note that $pb(\cU,\cV) \geq \max(pb(\cU),pb(\cV))$.
On the other hand, taking $\f=\g$ in the above definition, we get that $pb(\cU,\cU) \leq pb(\cU)$.
Therefore
\begin{equation}
\label{eq-pb-single-joint}
pb(\cU,\cU) = pb(\cU)\;.
\end{equation}

\medskip
\noindent The next result, which provides a lower bound on the noise indicator in terms of the Poisson
bracket invariants, serves as a bridge between quantum measurements and symplectic geometry.

\medskip
\noindent
\begin{thm}\label{thm-noisebound}
\begin{itemize}\item[{(i)}] For every finite open cover $\cU$ of $M$
\begin{equation}\label{eq-mu-pb-single}
\mu(\cU) \geq \frac{1}{2}\cdot pb(\cU)\;.
\end{equation}
\item[{(ii)}] For every pair of finite open covers $\cU$ and $\cV$ of $M$
\begin{equation}\label{eq-mu-pb-joint}
\mu(\cU,\cV) \geq \frac{1}{2}\cdot pb(\cU,\cV)\;.
\end{equation}
\end{itemize}
\end{thm}

\medskip
\noindent
\begin{proof} We shall prove \eqref{eq-mu-pb-joint} (the proof of \eqref{eq-mu-pb-single}
is completely analogous). Let $\cU$ and $\cV$ be a pair of  open covers
of a closed symplectic manifold $(M,\omega)$. Let $\f=\{f_i\}$, $\g=\{g_j\}$ be any partitions of unity subordinated to $\cU$ and $\cV$ respectively. Let $C^{(m)}$ be  any $\cL(H_m)$-valued POVM providing a joint measurement
for  $A^{(m)} =\{T_m(f_i)\}$ and $B^{(m)}=\{T_m(g_j)\}$. If $pb(\cU,\cV)=0$, \eqref{eq-mu-pb-joint}  follows automatically. Otherwise, take any  positive number
$p < pb(\cU,\cV)$. Observe that by \eqref{eq-pbquantbound}
$$\max \Big{(}\nu_q(A^{(m)}),\nu_q(B^{(m)}),\nu_q(A^{(m)},B^{(m)})\Big{)}$$
$$=\frac{1}{m}\cdot \max (\nu_c(\f),\nu_c(\g),\nu_c(\f,\g)) + O(1/m^2) \geq p/m$$
for all sufficiently large $m$. Hence $\cN_{in}(C^{(m)}) \geq p/(2m)$
for all sufficiently large $m$ by unsharpness principle \eqref{eq-joint}.
\end{proof}

\section{Poisson bracket invariants of fine covers}\label{sec-mainresults}

\subsection{Basic properties of the Poisson bracket invariants}
Recall that $\Omega_L$ stands for the finite set $\{1,...,L\}$.
Let $\cU= \{U_1,...,U_L\}$ and $\cW=\{W_1,...,W_K\}$ be two open covers of
$M$. The cover $\cW$ is called a {\it refinement} of $\cU$ if there exists
a map $\phi: \Omega_K \to \Omega_L$ such that $W_i \subset U_{\phi(i)}$ for all
$i \in \Omega_K$.

\medskip
\noindent
\begin{prop} \label{prop-refinements}
$pb(\cU) \leq pb(\cW)$ whenever $\cW$ is a refinement of $U$.
\end{prop}

\begin{proof}
Let $\g = \{g_1,...,g_K\}$ be any partition of unity subordinated to $\cW$.
For $l \in \Omega_L$ put
$$f_l = \sum_{i \in \phi^{-1}(l)} g_i\;.$$
The collection of functions $\f = \{f_l\}$ is a partition of unity subordinated to $\cU$.
Further, given $x \in K_L=[-1,1]^L$, the function
$\sum x_lf_l$ is a linear combination of $g_i$'s with coefficients from $[-1,1]$.
Thus $\nu_c(\f) \leq \nu_c(\g)$. Since $pb(\cU)\leq  \nu_c(\f)$, we get that
$pb(\cU) \leq \nu_c(\g)$ for every $\g$. Therefore $pb(\cU) \leq pb(\cW)$.
\end{proof}

\medskip
\noindent
For two open covers $\cU$ and $\cV$ denote by $\cU \cdot \cV$ the cover $\{U_i \cap V_j\}$.

\medskip
\noindent
\begin{prop}\label{pb-intersection}
$$ \max (pb(\cU), pb(\cV)) \leq pb(\cU \cdot \cV) \leq 4pb(\cU,\cV)\;.$$
\end{prop}

\begin{proof}
The inequality on the left immediately follows from Proposition \ref{prop-refinements}.
Let us prove the inequality on the right.

Take any partitions of unity $\f=\{f_i\}$ and $\g=\{g_j\}$ subordinated to coverings
$\cU$ and $\cV$ respectively. Put $$a:= \max\Big{(}\nu_c(\f),\nu_c(\g),\nu_c(\f,\g)\Big{)}\;.$$
Look at the partition of unity $\f \cdot \g= \{f_ig_j\}$ subordinated to the cover $\{U_i \cap V_j\}$ of $M$.
Note that
\begin{equation}\label{eq-brack-prod}
\{f_ig_j,f_kg_l\} = f_if_k\{g_j,g_l\} + g_jg_l \{f_i,f_k\} + f_ig_l\{g_j,f_k\} + g_jf_k\{f_i,g_l\}\;.
\end{equation}
Choose arbitrary weights $x_{ij}$ and $y_{kl}$ in $[-1;1]$. Put
$$F^x_j= \sum_i x_{ij}f_i\;, F^y_l = \sum_k y_{kl}f_k\;, G^x_i=\sum_j x_{ij}g_j\;, G^y_k = \sum_l y_{kl}g_l\;.$$
Fix a point $z \in M$ and put
$$I(z) = \Big{|} \Big{\{} \sum_{i,j}x_{ij}f_ig_j, \sum_{k,l}y_{kl} f_kg_l \Big{\}}(z) \Big{|}\;.$$
Calculating with the help of \eqref{eq-brack-prod} and estimating, we get that
$$I(z) \leq I_1(z) + I_2(z)+I_3(z)+I_4(z)\;,$$
where
$$I_1(z) = \sum_{i,k} f_i(z)f_k(z)|\{G^x_i,G^y_k\}(z)|\;\;,\;\; I_2(z)= \sum_{j,l} g_j(z)g_l(z)|\{F^x_j,F^y_l\}(z)|\;,$$
$$I_3(z) = \sum_{i,l} f_i(z)g_l(z) |\{G^x_i,F^y_l\}(z)|\;\;,\;\;
I_4(z) = \sum_{j,k}g_j(z)f_k(z)|\{F^x_j,G^y_k\}(z)|\;.$$
Observe that
$$|\{G^x_i,G^y_k\}(z)| \leq \nu_c(\{g_j\}) \leq a\;.$$
Thus
$$I_1(z) \leq a \cdot \sum_{i,k} f_i(z)f_k(z) = a \cdot \sum_i f_i(z) \cdot \sum_k f_k(z) =a\;.$$
Similarly,
$I_\alpha (z) \leq a$ for $\alpha=2,3,4$ and hence $I(z) \leq 4a$. Therefore
$$\nu_c(\f \cdot \g) \leq 4a\;.$$ This yields
$pb(\cU \cdot \cV) \leq 4pb(\cU,\cV)$, as required.
\end{proof}

\subsection{Small scales in symplectic geometry}
Recall that a {\it Hamiltonian diffeomorphism} is a time one map $\phi_H$ of a Hamiltonian flow on $M$ generated by a (in general, time dependent) Hamiltonian $H(z,t)$ on $M$. Hamiltonian diffeomorphisms
form a group denoted by $\Ham (M)$. Given a Hamiltonian diffeomorphism $\phi \in \Ham (M)$,
define its Hofer's norm \cite{Hofer} by
$$||\phi||_{Hofer} = \inf_{H\;:\:\phi_H=\phi} \int_0^1 \Big{(}\max_z H(z,t) -\min_z H(z,t) \Big{)}\cdot dt\;.$$
Here the infimum is taken over all Hamiltonians $H$ generating $\phi$, and hence Hofer's norm of $\phi$ measures the minimal possible amount of energy required in order to generate $\phi$.

A subset $Z \subset M$ is called {\it displaceable} if there exists a Hamiltonian diffeomorphism $\phi$ of $M$ so that
\begin{equation}\label{eq-displ}
\phi(Z) \cap \text{Closure}(Z) = \emptyset.
\end{equation}
For a displaceable subset $Z$ define its {\it displacement energy} $$E(Z)= \inf ||\phi||_{Hofer}\;,$$
where the infimum is taken over all $\phi \in \Ham(M)$ satisfying \eqref{eq-displ}.
We put $E(Z)=\infty$ if $Z$ is non-displaceable. We refer to \cite{Pbook} for an introduction to Hofer's geometry. The sets of displacement energy $\leq r$, $r \to 0$ provide a natural family of small scales in symplectic geometry.

In what follows, we shall need also relative versions of the above notions. Let $U \subset M$ be an
open subset. Denote by $\Ham(U)$ the subgroup of $\Ham(M)$ generated by Hamiltonian diffeomorphisms
$H(z,t)$ supported in $U \times [0;1]$. We say that $Z \subset U$ is {\it displaceable in $U$ }if
\eqref{eq-displ} holds for some $\phi \in \Ham(U)$, and we define the {\it relative displacement energy}
$E(Z,U)$ as the infimum of Hofer's norms of such $\phi$.

\subsection{Liouville domains} The following class of subsets will play an important role below. Let $\overline{U} \subset M^{2n}$ be a closed $2n$-dimensional submanifold with the interior $U$ and the boundary
$\partial U$.

\medskip
\noindent {\bf Topological assumption:} In what follows we assume that all connected components of $\overline{U}$ are simply connected. Let us emphasize that $\overline{U}$ is not assumed to be connected.

\medskip
\noindent
We say that $\overline{U}$ (resp. $U$) is {\it a closed (resp. open) Liouville domain} \cite{Eliashberg-Gromov}  if there exists a vector field $\xi$ on $\overline{U}$ which is transversal to the boundary $\partial U$ and which satisfies $L_\xi \omega = \omega$, where $L$ stands for the Lie derivative.  The vector field $\xi$ is called a Liouville vector field.

Note that at the boundary the field $\xi$ points outward. Therefore the flow $R_\tau: \overline{U} \to \overline{U}$ of $\xi$ is well defined for all non-positive times $\tau \leq 0$. This flow dilates the symplectic form: $R_\tau^* \omega = e^{\tau}\omega$.

Let $\xi$ and $\eta$ be two Liouville vector fields, and $R_\tau,T_\tau$, $\tau \leq 0$ be their flows.
Fix $\tau < 0$ and put $V= R_\tau(\overline{U})$ and $W= T_\tau(\overline{U})$. Observe that
$T_{-t}(R_{-t})^{-1}(V)$, $t \in [0, -\tau]$ is a symplectic isotopy which takes $V$ to $W$.
Since all connected components of $\overline{U}$ are simply connected, this isotopy is Hamiltonian.  By the Hamiltonian
isotopy extension theorem, this isotopy extends to a Hamiltonian diffeomorphism of $U$.
Since the displacement energy is invariant under Hamiltonian diffeomorphisms, the quantity $E(R_tU,U)$ does not depend on the specific choice of the Liouville field $\xi$.

In what follows it will be useful to rescale the variable $\tau$ and to adopt the following notation:
For $s \in (0,1]$ put $sU = R_{\log s}(U)$.

Observe that $s_1U \subset s_2U$ for $s_1 < s_2$.
The closed set $Q:= \bigcap_{s \in (0,1]} sU$ is called {\it the core} of $U$. We say that
the Liouville domain is {\it portable} (cf. \cite{BIP}) if $Q$ is displaceable in $U$.

For a portable Liouville domain $U$ define a function
$F_U(s) = E(sU,U)/s$. It is defined on $(0,1]$ and takes values in $(0,+\infty]$.
By the discussion above, $F_U$ is independent on the Liouville vector field.

Observe that $F_U(s)$ is finite for $s < s_0$ for some $s_0 >0$. Further,
$E(stU,U) \leq E(stU,tU) = tE(sU,U)$, where the inequality follows from the definitions,
and the equality from the fact that the pair $(U,sU)$ is conformally symplectomorphic to
$(tU,stU)$. Dividing by $st$, we get that $F_U(s)$ is non-decreasing on $(0,s_0]$.
Define the {\it portability number} of $U$ by
\begin{equation}\label{eq-portability-number}
\chi(U,\omega)= \inf F_U = \lim_{s\to 0} F_U(s)\;.
\end{equation}
Let us emphasize that the portability number is an intrinsic invariant of $(U,\omega)$: it does not
depend on the symplectic embedding $U \to M$. We put $\chi(U,\omega)=\infty$ if $U$ is not portable
and abbreviate $\chi(U)$ when the symplectic form $\omega$ is clear from the context.

\medskip
\noindent
\begin{exam}\label{exam-portable-1} {\rm Let $B^{2n}(r)$ be a closed Euclidean
ball of radius $r$ centered at $0$ in the linear symplectic space $(\R^{2n}, dp \wedge dq)$. This is a Liouville domain with the Liouville vector field $\xi = (p\partial/\partial p + q\partial/\partial q)/2$.
Its core coincides with the origin, and hence $U$ is portable. Observe that in our notations
$sB^{2n}(r) = B^{2n}(r\sqrt{s})$. For $s$ small enough, the displacement energy of  $B^{2n}(r\sqrt{s})$ in $B^{2n}(r)$ equals $\pi(r\sqrt{s})^2$ \cite{MS} and hence $\chi(B^{2n}(r))= \pi r^2$. }
\end{exam}

\begin{exam}\label{exam-starshaped}{\rm
More generally, every star-shaped bounded domain with smooth boundary in $\R^{2n}$ is portable Liouville. This yields that given a closed symplectic manifold $(M,\omega)$ equipped with a Riemannian metric, all metric balls of a sufficiently small radius are portable Liouville.
}\end{exam}

\medskip

Observe that the disjoint union of two closed Liouville domains is again a closed Liouville domain.
The following property of the portability number will be useful for our purposes.
Let $\overline{U}$ and $\overline{V}$ be a pair of disjoint closed Liouville domains. Since
$F_{U \sqcup V} = \max (F_U,F_V)$,
\begin{equation}\label{eq-sum-chi}
\chi(U \sqcup V)= \max(\chi(U),\chi(V))\;.
\end{equation}

\subsection{Main theorems on fine regular covers}\label{subsec-fine-regular}
We write $\overline{U}$ for the closure of a subset $U$.
Let $\cU=\{U_1,...,U_L\}$ be an open cover of a closed symplectic manifold
$(M,\omega)$. We say that {\it the degree} of $\cU$ is $\leq d$ if every subset $\overline{U}_j$
intersects closures of  at most $d$ other subsets from the cover. For a subset $X \subset M$
define its {\it star} $St(X)$ as the union of all $U_i$'s with $\overline{U}_i \cap \overline{X} \neq \emptyset$. The $p$ times iterated
star (for brevity, $p$-star) $St(...(St(X)...)$ is denoted by $St_p(X)$. Put
$$E_p(\cU) = \max_{i} E(U_i, St_p(U_i))\;.$$  We say that the cover $\cU$ is {\it $(d,p)$-regular} if the degree of $\cU$ is $\leq d$ and $E_p(\cU) <\infty$. The latter condition means that every subset $U_j$ is displaceable in its $p$-star. Let us emphasize that the notion of $(d,p)$-regularity and as well as the
quantity $E_p$ is invariant under symplectomorphisms of $M$.

\medskip
\noindent
\begin{exam}\label{exam-dpregular-greedy}{\rm
Let $\cU$ be a greedy cover responsible for the phase space discretization into Riemannian balls of radius $r$, see Example \ref{exam-greedy} above. For all $0< r \leq r_0$ this cover is $(d,p)$-regular with
$E_p(\cU) \leq \kappa r^2$, where  the constants $r_0,d,p$ and $\kappa$ depend only on the symplectic manifold $(M,\omega)$ and the Riemannian metric $\rho$. The proof will be given in Section \ref{sec-dpregular-greedy}.}
\end{exam}

\medskip
\noindent
The moral of this example is that for some positive integers $d,p$ the following holds true:
for every $\epsilon >0$ every finite open cover $\cV$ of $M$ admits a $(d,p)$-regular refinement
with $E_p < \epsilon$. In view of this we shall refer to $E_p(\cU)$ as the magnitude of localization
of $\cU$. The magnitude of localization could be arbitrarily small while the parameters
$d,p$ remain bounded.

An important property of $(d,p)$-regular covers is as follows.

\medskip
\noindent\begin{prop}\label{prop-dpregular-refinement}
Every $(d,p)$-regular cover $\cU$ is a refinement of a cover $\cW=\{W_1,...,W_N\}$
where $N$ depends only on $d,p$ and $E(W_j) \leq E_p(\cU)$ for all $j$.
\end{prop}

\medskip

Let us pass now to fine open covers $\cU$ by portable Liouville domains. In this case we
keep track of the degree of the cover and measure the magnitude of localization
with the help of the quantity $\chi(\cU):= \max_j \chi(U_j)$. Again, greedy covers of a sufficiently
small radius provide an example of a Liouville cover of bounded degree with an arbitrary small
magnitude of localization (see Example \ref{exam-starshaped} above).

\medskip
\noindent \begin{prop}\label{prop-Liouville-refinement}
Every degree $\leq d$ cover $\cU$ by portable Liouville domains is a refinement of a cover $\cW=\{W_1,...,W_N\}$ by portable Liouville domains, where $N \leq d+1$ and $\chi(\cU) = \chi(\cW)$.
\end{prop}

\medskip
\noindent
The main result of the present section is as follows:

\medskip
\noindent \begin{thm}\label{thm-main-fine covers}
Let  $\cW=\{W_1,...,W_N\}$ be an open cover of a closed symplectic manifold $(M,\omega)$.
Assume that
\item[{(i)}] either all subsets $W_j$ are displaceable with $E(W_j)\leq \cA$.
\item[{(ii)}] or $\pi_2(M)=0$, and all subsets $W_j$ are Liouville with
$\chi(W_j) \leq \cA$.
Then
\begin{equation} \label{eq-pb-fine-displ}
pb(\cW) \geq C(N)\cA^{-1}\;,
\end{equation}
where the constant $C(N)$ depends only on $N$.
\end{thm}

\medskip
\noindent In (ii) $\pi_2(M)$ stands for the second homotopy group of $M$. The condition
$\pi_2(M)=0$ means the every sphere in $M$ contracts to a point. It holds, for instance,
when $M$ is the $2n$-dimensional torus. This condition  cannot be lifted: one can show
that Theorem \eqref{thm-main-fine covers}(ii) is wrong as stated for the sphere $S^2$.
It would be interesting to explore whether the statement could be modified so that it
will hold for all closed symplectic manifolds.

\medskip
\noindent In view of Propositions \ref{prop-dpregular-refinement} and \ref{prop-Liouville-refinement}
we get that for a $(d,p)$-regular cover (resp., a Liouville cover of degree $\leq d$) with the magnitude
of localization $\cA$
\begin{equation} \label{eq-mag-loc}
pb(\cU) \geq C\cdot\cA^{-1}\;,
\end{equation}
where the constant $C$ depends only on $d,p$ (resp. only on $d$). By Proposition \ref{prop-refinements}
the same inequality holds for refinements of these covers.

\subsection{Combinatorics of covers}
Our next goal is to prove Propositions \ref{prop-dpregular-refinement} and \ref{prop-Liouville-refinement}.
It would be convenient to introduce the following combinatorial language.
Let $\cU=\{U_1,...,U_L\}$ be an open cover of
a closed manifold $M$. Define the graph $\Upsilon$ whose vertices form the set
$\Omega_L =\{1,...,L\}$. Vertices $i$ and $j$ are joined by an edge if and only if
$\overline{U}_i \cap \overline{U}_j \neq \emptyset$.
Let $\gamma$ be the graph metric on $\Upsilon$: the distance between two vertices
is defined as the number of edges in the shortest path connecting them.
Denote by $d(\cU)$ the maximal degree of a vertex in $\Upsilon$.

Recall that the chromatic number of a graph is the smallest number of colors needed to color the vertices  so that no two adjacent vertices share the same color. It is well known that the chromatic number does not exceed
$d+1$ where $d$ is the maximal vertex degree of the graph.

\medskip
\noindent
\begin{prop}\label{prop-colors} For every $k \in \N$ the vertices of $\Upsilon$ can be colored into $\leq d(\cU)^k+1$
colors in such a way that every two vertices of the same color lie at the distance
$\geq k+1$.
\end{prop}

\begin{proof} Denote by $\Upsilon^k$ a graph whose vertices coincides with the ones of $\Upsilon$,
and where two vertices are connected by an edge if and only if the distance between them is $\leq k$.
Clearly, the maximal degree of a vertex in  $\Upsilon^k$ does not exceed $a:= d(\cU)^k$. Therefore
its chromatic number does not exceed $a+1$. By definition, this means that $\Upsilon^k$ can be colored
in $\leq a+1$ colors so that any two vertices $i,j$ of the same color are not connected by an edge
in $\Upsilon^k$. This means the $\gamma(i,j) \geq k+1$.
\end{proof}

\medskip

{\bf Proof of Proposition \ref{prop-dpregular-refinement}:}
Let $\cU=\{U_1,...,U_L\}$ be a $(d,p)$-regular open cover of $M$. By Proposition \ref{prop-colors}
the sets $U_j$ can be colored in $N \leq d^{2p}+1$ colors so that for every two sets
$U_i$ and $U_j$ of the same color the stars $St_p(U_i)$ and $St_p(U_j)$ do not intersect.
For $\alpha \in \{1,...,N\}$
denote by $I_\alpha$ the set of all indices $i$ such that $U_i$ is of color $\alpha$.
Put $W_\alpha = \sqcup_{i \in I_\alpha} U_i$. The cover $\cU$ is a refinement of $\{W_\alpha\}$.
Since for all $i \in I_\alpha$ the set $U_i$ is displaceable in $St_p(U_i)$ and the stars
$St_p(U_i)$ are pairwise disjoint,
$$E(W_\alpha) \leq E(\sqcup_{i \in I_\alpha} U_i, \sqcup_{i \in I_\alpha} St_p(U_i))=$$
$$\max_{i \in I_\alpha}E(U_i, St_p(U_i)) \leq E_p(\cU)\;.$$
This completes the proof.
\qed

\medskip
\noindent
{\bf Proof of Proposition \ref{prop-Liouville-refinement}:}
Let $\cU= \{U_1,...,U_L\}$ be a $d$-regular finite open cover of $M$
so that all subsets $U_j$ are open portable Liouville domains.
The elements of the cover can be colored
in $N \leq d+1$ colors in such a way that for every two subsets $U_i,U_j$
of the same color $\overline{U_i}\cap \overline{U_j} = \emptyset$. For $\alpha \in \{1,...,N\}$
denote by $I_\alpha$ the set of all indices $i$ such that $U_i$ is of color $\alpha$.
Put $W_\alpha = \sqcup_{i \in I_\alpha} U_i$. The cover $\cU$ is a refinement of $\{W_\alpha\}$. Since $W_\alpha$ is the disjoint union of Liouville domains,
it is again a Liouville domain with $$\chi(W_\alpha) = \max_{i \in I_\alpha} \chi(U_i)\;,$$
so that $\chi(\cW)=\chi(\cU)$.
This completes the proof.
\qed

\subsection{Regularity of greedy covers}\label{sec-dpregular-greedy}
In this section we justify the claim of Example \ref{exam-dpregular-greedy}.
Let $\cU$ be a greedy cover responsible for the phase space discretization into Riemannian balls of radius $r$, see Example \ref{exam-greedy} above. We identify the vertices of the graph $\Upsilon$ with the set
$Z$ consisting of the centers of the balls. With this language $z,w \in Z$ are connected by an edge if and
only if $\rho(z,w) \leq 2r$. For a vertex $z \in Z$ we denote by $St_k(z) \subset M$ its $k$-star.
We tacitly assume that all the distances appearing below are small
in comparison with the injectivity radius of the Riemannian metric $\rho$.

\medskip
\noindent
{\sc Step 1:} Fix $k \in \N$. We claim that $D(z,kr) \subset St_k(z)$ for all
sufficiently small $r$. Indeed, suppose that $\rho(z,w) \leq kr$, where $k \in \N$ and $z \in Z$.
Joining $z$ and $w$ with the shortest geodesic and partitioning,
we get points $$u_0=z,u_1,...,u_{k-1}, u_k=w$$
with $\rho(u_i,u_{i+1}) \leq r$ for all $i=0,...,k-1$.
Each points $u_i$ lies at the distance $< r/2$ from some point $u'_i \in Z$,
where we put $u'_0=z$.
By the triangle inequality, $\rho(u'_i,u'_{i+1}) < 2r$ and hence the graph distance
$\gamma(z,u'_k)$ is $\leq k$. Since $w \in D(u'_k,r)$ we conclude that $w \in St_k(z)$.
The claim follows.

\medskip
\noindent {\sc Step 2:} We claim that the degree of every vertex $z \in Z$ in the graph $\Upsilon$
is $\leq d$ where $d$ depends only on the Riemannian metric. Indeed,
let $v(r)$ (resp. $V(r)$) be the minimal (resp. the maximal)
volume of the Riemannian ball of radius $r$ on $M$. Put
$$b=\sup_{0 < r < r_0} V(3r)/v(r/4)\;.$$
Let $d_z$ be the degree of a vertex $z \in Z$ in the graph $\Upsilon$ and $Y$ be the set of neighbors of $z$.
Since the distance between each two of the neighbors is $> r/2$,
the balls $ D(y,r/4)$, $y \in Y$ are pairwise disjoint. At the same time $\rho(z,y) \leq 2r$
for all $y \in Y$ and hence $D(y,r/4) \subset D(z,3r)$. Therefore,
$$\sqcup_{y\in Y} D(y,r/4) \subset D(z,3r)\;.$$
It follows that $d_z\cdot v(r/4) \leq V(3r)$ and hence $d \leq b$.

\medskip
\noindent
{\sc Step 3:} Choose $r_1>0 $ small enough such that for every $j$ and $0 < r < r_1$ the ball $D(z,r)$ lies in a Darboux chart of $(M,\omega)$. We can assume that for every $z \in Z$ this chart is identified with the standard symplectic ball $B(r_2) \subset \R^{2n}$, where $z$ corresponds to the origin $0 \in \R^{2n}$. Furthermore, for some positive constants $a < b$
$$B(ar) \subset D(z,r) \subset B(br)\;\; \forall r \in (0,r_1), z \in Z \;.$$
In particular, we have that for $r$ small enough
$$D(z,r) \subset B(br) \subset B(10br) \subset D(z,cr), \;\; c=10b/a\;.$$
Fix an integer $p > c$. By Step 1 we have that
$D(z,cr) \subset St_p(z)$. Comparing the displacement energies we get that
$$E(D(z,r), St_p(z)) \leq E(B(br), B(10br)) = \pi b^2r^2\;.$$
Therefore the greedy cover $\cU$ is $(d,p)$-regular with $E_p(\cU) \leq \pi b^2r^2$.
This proves the claim of Example \ref{exam-dpregular-greedy}.
\qed

\section{Symplectic geometry of fine covers}\label{sec-symptop}

The rest of this section is dedicated to the proof of Theorem \ref{thm-main-fine covers}.

\subsection{Algebraic preliminaries}\label{subsubsec-alg}
Let us recall algebraic preliminaries which are borrowed from \cite{EP-qss}.
Let $\cG$ be a group, and let $c: \cG \to \R$ be a function with the following properties:
\begin{itemize}
\item[{(i)}] $c(1)=0$;
\item[{(ii)}] $c(\phi\psi\phi^{-1})=c(\psi) \;\; \forall \phi,\psi \in \cG$;
\item[{(iii)}] $c(\phi\psi) \leq c(\phi)+c(\psi)$.
\end{itemize}
We start with the following two observations. First, the function
$$q(\phi):= c(\phi) + c(\phi^{-1})$$
is a pseudo-norm on $G$. This means, by definition, that
$q(1)=0$, $q$ is symmetric: $q(\phi)=q(\phi^{-1})$, and $q$ satisfies the triangle inequality
$q(\phi\psi) \leq q(\phi)+q(\psi)$. In addition, $q$ is conjugation invariant: $q(\phi\psi\phi^{-1})=q(\psi)$.

Second, we claim that for every $\phi \in \cG$ the limit
$$\sigma(\phi):= \lim_{k \to \infty} c(\phi^k)/k$$
is well defined. Indeed, put
$a_k= kc(\phi^{-1}) + c(\phi^k)$. Then $a_k \geq c(\phi^{-k})+ c(\phi^k) \geq 0$ and
$a_{k+l}\leq a_k+a_l$. The subadditivity implies that the limit $\lim_{k\to \infty} a_k/k$ is well defined,
and hence  $\sigma(\phi)$ is well defined. The claim follows.

\medskip
\noindent
\begin{prop}\label{prop-algebraic}
\begin{equation}\label{eq-alg}
|\sigma(\phi\psi)-\sigma(\phi)-\sigma(\psi)| \leq \min(q(\phi),q(\psi))
\end{equation}
for all $\phi,\psi \in \cG$.
\end{prop}
\begin{proof} We start with the formula
\begin{equation}\label{eq-product}
(\phi\psi)^k = \theta\phi^k,\;\;\;\text{where} \;\;\; \theta=  \prod_{i=1}^k \phi^i\psi\phi^{-i} \;.
\end{equation}

\medskip
\noindent 1) Note that
$$kc(\psi^{-1}) + c(\psi^k) \geq c(\psi^{-k}) + c(\psi^k) \geq 0\;,$$
and hence
$$kc(\psi)=kq(\psi)-kc(\psi^{-1}) \leq kq(\psi) + c(\psi^k)\;.$$
By \eqref{eq-product}
$$c((\phi\psi)^k) \leq c(\theta) + c(\phi^k) \leq kc(\psi)+c(\phi^k) \leq kq(\psi)+c(\psi^k)+c(\phi^k)\;.$$
Dividing by $k$ and passing to the limit as $k \to \infty$ we get that
\begin{equation}\label{eq-algebraic-leq}
\sigma(\phi\psi)-\sigma(\phi)-\sigma(\psi) \leq q(\psi)\;.
\end{equation}

\medskip
\noindent
2) Note that
$$kc(\psi^{-1}) = -kc(\psi) + kq(\psi) \leq -c(\psi^k) + kq(\psi)\;.$$
By \eqref{eq-product},
$$c(\phi^k) \leq c(\theta^{-1}) + c((\phi\psi)^k) \leq kc(\psi^{-1})+  c((\phi\psi)^k)\;.$$
Combining these inequalities we get that
$$c(\phi^k) \leq -c(\psi^k) + kq(\psi) + c((\phi\psi)^k)\;.$$
Rearranging the terms, dividing by $k$ and passing to the limit as $k \to \infty$
we get that
$$\sigma(\phi\psi)-\sigma(\phi)-\sigma(\psi) \geq -q(\psi)\;.$$
Together with \eqref{eq-algebraic-leq} this yields
$$|\sigma(\phi\psi)-\sigma(\phi)-\sigma(\psi)| \leq q(\psi)\;.$$
Since $c$ is conjugation invariant, $\sigma$ is also conjugation invariant and thus
$\sigma(\phi\psi)=\sigma(\psi\phi)$. Repeating the above arguments we get that
$$|\sigma(\phi\psi)-\sigma(\phi)-\sigma(\psi)| \leq q(\phi)\;.$$
Together with the previous inequality this yields \eqref{eq-alg}.
\end{proof}

\subsection{Spectral invariants}\label{subsec-spectral} Let $(M,\omega)$ be a closed symplectic manifold.
Denote by $\cF$ the space of all smooth functions $F: M \times [0,1] \to \R$. Given such an $F$,
we write $F_t(z):= F(z,t)$. For an open subset $U \subset M$ introduce the subspace $\cF(U)$ consisting of all $F \in \cF$ such that $F_t$ is supported in $U$ for all $t \in [0;1]$.
Let $\cG$ be the  universal cover of the group of Hamiltonian diffeomorphisms of
$(M,\omega)$. Given a Hamiltonian $F \in \cF$, we write $\phi^t_F$ for the Hamiltonian flow generated by $F$,
and $\phi_F$ for the element of $\cG$ represented by the path $\{\phi^t_F\}$, $t \in [0;1]$.
Given an open subset $U \in M$, Hamiltonians from $\cF(U)$ generate a
subgroup $\cG(U) \subset \cG$.

In what follows we are going to use the technique of spectral invariants
introduced in \cite{Schwarz, Oh}. We denote by $c: \cG \to \R$ the spectral invariant associated to the fundamental class of $M$. The function $c$ satisfies properties (i)-(iii) from Section \ref{subsubsec-alg}
above. We keep the notations $q(\phi)= c(\phi)+c(\phi^{-1})$ and
$\sigma(\phi)= \lim_{k \to \infty} c(\phi^k)/k$ introduced in this section.
For a function $F \in \cF$ define its mean value as
$$\langle F \rangle = \frac{\int_0^1 \int_M F \; \omega^n \cdot dt}{ \int_M \omega^n }\;,$$
where $\dim M =2n$. Introduce the {\it partial symplectic quasi-state} \cite{EP-qss}
$\zeta: C^{\infty}(M) \to \R$ by
$$\zeta(F) = \sigma(\phi_F) + \langle F \rangle\;.$$

Let us list some additional properties of these functionals which will be used below (see \cite{EP-qss,EPZ}):

\medskip
\noindent
{\sc Stability:}
\begin{equation} \label{eq-Lipschitz-c}
|(c(\phi_F)+ \langle F \rangle)- (c(\phi_G)+ \langle G \rangle)| \leq \int_0^1||F_t-G_t||\; dt\;
\end{equation}
and
\begin{equation} \label{eq-Lipschitz}
|(\sigma(\phi_F)+ \langle F \rangle)- (\sigma(\phi_G)+ \langle G \rangle)| \leq \int_0^1||F_t-G_t||\; dt\;.
\end{equation}

\medskip
\noindent {\sc Homogeneity:} The functional $\zeta$ is $\R_+$-homogeneous:  $\zeta(sF) = s\zeta(F)$ for all $s > 0$.

\medskip \noindent {\sc Vanishing:} A key property of $\zeta$ is that $\zeta(F)=0$ provided the support of $F$ is displaceable. Further (see \cite{EP-qss})
\begin{equation}\label{eq-displacement-spectral}
q(\phi) \leq 2E(U) \;\; \forall \phi \in \cG(U)\;,
\end{equation}
where $E(U)$ is the displacement energy of $U$.

\medskip
\noindent {\sc Spectrality (see \cite{U}):} Finally, let us mention that the number $c(\phi)$ necessarily lies in {\it the action spectrum of $\phi$},
a special subset associated to a (lift of a) Hamiltonian diffeomorphism $\phi \in \cG$ as follows: Represent $\phi$ as $\phi_F$ for some Hamiltonian $F \in \cF$. Let $\gamma = \{\phi_F^t x\}$, $t \in [0,1]$ be a contractible closed orbit of the Hamiltonian flow of $F$, that is $\phi_F^1 x =x$. Choose any two-dimensional disc $D \subset M$ spanning $\gamma$.
By definition,
$$\text{Action}(\gamma,D)= \int_0^1 F( \phi_F^t x, t) \; dt-\int_D \omega\;.$$
The action spectrum $\text{spec}(\phi) \subset \R$ is defined as the set of actions of all closed orbits
of $\phi$ with respect to all spanning discs $D$. The action spectrum does not depend on the
specific Hamiltonian $F$ generating $\phi$ and has empty interior in $\R$ (see \cite[Lemma 2.2]{Oh-Asian}).
Moreover, if $\pi_2(M)=0$, all discs $D$ spanning a given orbit $\gamma$ are homotopic with fixed boundary,
so that $\int_D \omega$ does not depend on $D$. In this case $\text{spec}(\phi)$ is a compact subset of $\R$.

\subsection{The Poisson bracket inequality}
For $F,G \in C^{\infty}(M)$ put
$$\Pi(F,G) = |\zeta(F+G)-\zeta(F)-\zeta(G)|\;\;\text{and}\;\;S(F,G) = \sup_{s >0} \min (q(\phi_{sF}),q(\phi_{sG}))\;.$$
Observe that by \eqref{eq-displacement-spectral}
\begin{equation}\label{eq-displacement-spectral-1}
S(F,G) \leq 2\min (E(\text{supp} F),E(\text{supp} G))\;.
\end{equation}
The next proposition is a modification of a result from \cite{EPZ}.

\medskip
\noindent
\begin{prop}\label{prop-pbinequality}
For every functions $F,G \in C^{\infty}(M)$
\begin{equation}\label{eq-pbinequality}
\Pi(F,G) \leq \sqrt{2S(F,G) \cdot  ||\{F,G\}|| }.
\end{equation}
\end{prop}
\begin{proof}
Let $f_t, g_t$ be the Hamiltonian flows of $F,G$, respectively. Set
$
H = F+G$, $K_t = F + G \circ f_t^{-1}$.
Observe that $\phi_K = \phi_F \, \phi_G$. Calculating
$$
G(f_t x) - G(x) = \int_0^t \frac{d}{ds} G(f_s x) ds = \int_0^t \{G, F\}(f_s x) ds\;,
$$
we get that
$
|| H-K_t || = ||G - G \circ f_t^{-1} || = ||G \circ f_t - G || \leq t || \{F,G\}||
$.
Since $\langle K \rangle = \langle F+G \rangle$,
we have by \eqref{eq-Lipschitz} that
$$
|\sigma(\phi_{F+G}) - \sigma(\phi_F \, \phi_G)| \leq \int_0^1 || H-K_t ||\; dt \leq \frac{1}{2} ||\{F,G\} ||.
$$
Then
\begin{multline}\label{4-ineq-Pi-unbalanced}
\Pi(F,G) = | \sigma(\phi_{F+G}) - \sigma(\phi_F) - \sigma(\phi_G)| \leq |\sigma(\phi_{F+G}) - \sigma(\phi_F \, \phi_G)|  \\
 + |\sigma(\phi_F \, \phi_G) - \sigma(\phi_F) - \sigma(\phi_G)|
\leq \frac{1}{2} ||\{ F,G\} || + I\;,
\end{multline}
where $I:= |\sigma(\phi_F \, \phi_G) - \sigma(\phi_F) - \sigma(\phi_G)|$.
By Proposition \ref{prop-algebraic},
$I \leq S(F,G)$,
and therefore
\begin{equation}
\label{4-ineq-Pi-unbalanced-1}
\Pi(F,G) \leq   \frac{1}{2} ||\{ F,G\} || + S(F,G)\;.
\end{equation}
Let us balance this inequality:  Take any $s > 0$, and note that
$$
\Pi(sF, sG) = s \cdot \Pi(F,G) ,\,\, || \{sF, sG\} || = s^2 ||\{F,G\} ||,\; S(sF,sG)=S(F,G)\;.
$$
Substituting into \eqref{4-ineq-Pi-unbalanced-1} and dividing by $s$ we obtain
$$
\Pi(F,G) \leq \frac{1}{2} s ||\{F,G\} || + \frac{S(F,G)}{s}.
$$
The right hand side is minimized by $s = \sqrt{2S(F,G) / ||\{F,G\} ||}$, which yields
$$
\Pi(F,G) \leq \sqrt{2S(F,G) \cdot  ||\{F,G\}|| }.
$$
\end{proof}

\medskip
\noindent
\begin{cor}\label{cor-nu-displacement} Let $\cW=\{W_1,...,W_N\}$ be an open cover of $M$ by displaceable
open subsets. Assume that for every $j$
\begin{equation} \label{eq-supq}
\sup_{\phi \in \cG(W_j)} q(\phi) \leq Q\;.
\end{equation}
Then
\begin{equation}\label{eq-pb-displaceable-cor}
pb(\cW) \geq   1/(2N^2Q)\;.
\end{equation}
\end{cor}
\begin{proof} Let $\f=\{f_1,...,f_N\}$ be any partition of unity subordinated to $\cW$.
Observe that $\zeta(f_i)=0$ since the support of $f_i$ is displaceable.
Put $g_k  = f_1+...+f_k$.
By \eqref{eq-pbinequality},
$$\zeta(g_{k+1}) \leq \zeta(g_k) + \sqrt{2Q\nu_c(\f)}\;.$$
Thus
$$1=\zeta(g_N) \leq N\sqrt{2Q\nu_c(\f)}\;,$$
and hence
$\nu_c(\f) \geq 1/(2N^2Q)$. This yields \eqref{eq-pb-displaceable-cor}.
\end{proof}

\subsection{Proof of main theorems on fine regular covers}

\medskip
\noindent
{\bf Proof of Theorem \ref{thm-main-fine covers}(i):}
Let $(M,\omega)$ be a closed symplectic manifold, and let $\cW=\{W_1,...,W_N\}$
be an open cover of $M$ such that
all subsets $W_j$ are displaceable with $E(W_j)\leq \cA$.
Applying Corollary \ref{cor-nu-displacement}
combined with formula  \eqref{eq-displacement-spectral}
to the cover $\cW$ we get that
$$pb(\cW) \geq   C(N) \cdot \cA^{-1}\;,$$ the positive constant $C(N)$ depends only on $N$.
\qed

\medskip
\noindent In order to prove the second part of the main theorem, we shall need the following
auxiliary statement.

\medskip
\noindent
\begin{prop}\label{prop-Liouville-spectral} Let $(M,\omega)$ be a closed symplectic manifold with
$\pi_2(M) =0$. Let $\overline{W} \subset M$ be a portable closed Liouville domain.
Then for every $\phi \in \cG(W)$
\begin{equation}
\label{eq-disp-chi}
q(\phi) \leq 2\chi(W)
\end{equation}
\end{prop}

\medskip
\noindent Recall that according to our conventions, $W$ is the interior of $\overline{W}$ and
all connected components of $\overline{W}$ (and hence of $W$) are simply connected.

\medskip
\noindent
\begin{proof} Let $R_\tau$, $\tau < 0$ be the flow of the Liouville vector field.
Take any $u \in (0,1)$ such that $uW$ is displaceable
in $W$ with the displacement energy $E_u$.
Suppose that $\phi=\phi_F$, where $F \in \cF(W)$. Write $f_t$ for the Hamiltonian flow of $F$ so that
$f_1 =\phi$. Define a path of Hamiltonian diffeomorphisms $f_t^{(s)}$, $s \in [u,1]$ as
$R_{\log s}f_tR_{\log s}^{-1}$ on $sW$ and as the identity outside $sW$. Observe that for every $s$
the path $\{f_t^{(s)}\}$ is generated by the Hamiltonian
\begin{equation}\label{eq-ham-dilation}
F_s(z,t) = sF(R_{\log s}^{-1}z,t)\;.
\end{equation}
Put $\phi^{(s)} = f_1^{(s)}$.

Consider the following subset of the plane $\R^2$ equipped with
the coordinates $(s,a)$: for every $s \in [u,1]$ mark on the line $s \times \R$ the action spectrum
of $\phi^{(s)}$. We claim that this set is a disjoint union of segments lying on straight lines
passing through the origin. Indeed,
$\gamma$ is a contractible closed orbit of $\{f_t\}$ if and only if $\gamma_s:= R_{\log s}\gamma$ is a
contractible closed orbit of  $\{f_t^{(s)}\}$. Since $\pi_2(M)=0$, the action of a contractible closed
orbit does not depend on the choice of the spanning disc. Since $W$ is simply connected,
$\gamma$ can be spanned by a disc $D \subset W$, and hence $\gamma_s$ can be spanned by $D_s:=R_{\log s} D$. Taking into account
that $\int_{D_s}\omega = s\cdot \int_\omega D$ and by using \eqref{eq-ham-dilation} we get that
$$\text{Action}(\gamma_s,D_s) = s \cdot \text{Action}(\gamma,D)\;.$$
This yields the claim. Remembering that the action spectrum is nowhere dense and that $c(\phi^{(s)}) \in \text{spec}(\phi^{(s)})$ depends continuously on $s$, we conclude that
$c(\phi^{(s)})= sc(\phi)$.
Similarly, $c((\phi^{(s)})^{-1})= sc(\phi^{-1})$, and so
$q(\phi^{(s)})= sq(\phi)$ for all $s \in [u,1]$.
By \eqref{eq-displacement-spectral}, $q(\phi^{(u)}) \leq 2E_u$ and
hence $q(\phi) \leq 2E_u/u$. Passing to the limit as $u \to 0$, we get that
$q(\phi) \leq 2\chi(W)$, as required.
\end{proof}

\medskip
\noindent
{\bf Proof of Theorem \ref{thm-main-fine covers}(ii)}
Let $(M,\omega)$ be a closed symplectic manifold with $\pi_2(M)=0$.
Consider a cover $\cW=\{W_1,...,W_N\}$ such that
all subsets $W_j$ are Liouville with
$\chi(W_j) \leq \cA$. By Proposition \ref{prop-Liouville-spectral}
\begin{equation} \label{eq-supq-1}
\sup_{\phi \in \cG(W_j)} q(\phi) \leq 2\cA\;
\end{equation}
for every $j=1,...,N$. Applying Corollary \ref{cor-nu-displacement}
to the cover $\cW$ we get that
$pb(\cW) \geq C(N)\cA^{-1}$.
\qed

\section{Geometry of overlaps }\label{subsec-pb4-overlaps}

\subsection{$pb_4$-invariant}\label{subsec-pb4}
We start with the definition of the Poisson bracket invariant $pb_4$, which was introduced and studied in \cite{BEP}. Let $X_0,X_1,Y_0,Y_1$ be a quadruple of compact subsets of a symplectic manifold $(M,\omega)$ satisfying
an intersection condition
$$X_0 \cap X_1 = Y_0 \cap Y_1 =\emptyset\;.$$
One of the equivalent definitions of $pb_4$ is as follows, see
\cite[Proposition 1.3]{BEP}:
$$pb_4(X_0,X_1,Y_0,Y_1)= \inf ||\{f,g\}||\;,
$$
where the infimum is taken over all pairs of functions $f,g: M \to [0;1]$ with
$f=0$ near $X_0$, $f=1$ near $X_1$, $g=0$ near $Y_0$ and $g=1$ near $Y_1$.

\medskip
\noindent \begin{exam}\label{exam-pb4-quadrilateral}{\rm
Consider the unit sphere $S^2$ equipped with the standard area form.
Let $\Pi \subset S^2$ be a quadrilateral whose edges (in the cyclic order) are denoted by  $u_1,v_1,u_2,v_2$. By \cite[Theorem 1.20]{BEP}
\begin{equation}\label{eq-pb4-quadrilateral}
pb_4(u_1,u_2,v_1,v_2) = \max\Big{(}1/\text{Area}(\Pi),1/(\text{Area}(S^2)- \text{Area}(\Pi)\Big{)}\;.
\end{equation}}\end{exam}

\medskip
\noindent The proof of \eqref{eq-pb4-quadrilateral} presented in \cite{BEP} is elementary.
Let us mention, however,  that proving positivity of $pb_4$ in specific examples on symplectic manifolds of dimension $\geq 4$ involves ``hard" symplectic techniques such as theory of symplectic quasi-states
(cf. Section \ref{subsec-spectral} above) and/or the Donaldson-Fukaya category (see \cite{BEP}). Interestingly enough, both tools have strong links to physics: the former is related to a discussion on non-contextual hidden variables in quantum mechanics, while the latter comes from mirror symmetry.

In what follows we shall need the following obvious monotonicity property of $pb_4$ (see \cite[Section 2]{BEP}:
Assume that subsets $X'_0,X'_1,Y'_0,Y'_1$ are contained in  $X_0,X_1,Y_0,Y_1$ respectively.
Then
\begin{equation}\label{eq-pb4-monotone}
pb_4(X'_0,X'_1,Y'_0,Y'_1) \leq pb_4(X_0,X_1,Y_0,Y_1)\;.
\end{equation}
Furthermore, $pb_4(X_0,X_1,Y_0,Y_1)$ is invariant under permutations $X_0$ with $X_1$, $Y_0$ with $Y_1$
and of the pairs $(X_0,X_1)$ and $(Y_0,Y_1)$.

\medskip
\noindent For a subset $U \subset M$ denote by $U^c:= M \setminus U$ its set-theoretic complement.
Let $\cU= \{U_1,...,U_L\}$ be an open cover
of a topological space $M$. Given a subset $I \subset \Omega_L=\{1,...,L\}$, put
\begin{equation}\label{eq-UI}
U(I):=\Big{(}\bigcup_{\alpha \in I} U_\alpha\;\Big{)^c}\;,
\end{equation}
and
\begin{equation}\label{eq-lambdaI}
\Lambda(\cU,I)= \bigcup_{\alpha \in I, \beta \in I^c} U_\alpha \cap U_\beta\;.
\end{equation}
The subset $U(I)$ is closed and $\Lambda(\cU,I)$ is open.
The manifold $M$ is decomposed into the union of three disjoint subsets
\begin{equation}\label{eq-decomposition}
M = U(I) \sqcup \Lambda(\cU,I) \sqcup U(I^c)\;.
\end{equation}
We call $\Lambda(\cU,I)$ {\it an overlap layer} of the cover $U$.

\medskip
\noindent
The main result of the present section is as follows.

\medskip
\noindent
\begin{thm}\label{thm-overlap-layer-pb4}
Let $\cU=\{U_1,...,U_L\}$ and $\cV= \{V_1,...,V_N\}$ be two finite
open covers of $M$. Then for every $I \subset \Omega_L$ and $J \subset \Omega_N$
\begin{equation}\label{eq-mu-pb4-overlap}
pb(\cU,\cV) \geq 4pb_4(U(I),U(I^c),V(J),V(J^c))\;.
\end{equation}
\end{thm}

\medskip
\noindent In view of decomposition \eqref{eq-decomposition}, the value
of $pb_4(U(I),U(I^c),V(J),V(J^c))$ depends only on the relative position of the overlap layers
$\Lambda(\cU,I)$ and $\Lambda(\cV,J)$. Note also, that since $pb(\cU,\cU)=pb(\cU)$ by
\eqref{eq-pb-single-joint}, the theorem is applicable to a single cover. It implies that for all $I,J$
\begin{equation}\label{eq-mu-pb4-overlap-single}
pb(\cU) \geq 4pb_4(U(I),U(I^c),U(J),U(J^c))\;.
\end{equation}

\subsection{Two-sets covers}
As an illustration, consider a pair of open covers consisting
of exactly two sets, $\cU=\{U_1,U_2\}$ and $\cV=\{V_1,V_2\}$. The corresponding partitions
of unity necessarily have the form $f,1-f$ and $g,1-g$.
It turns out that in this case the Poisson bracket invariant $pb$ coincides (up to a multiple)
with the invariant $pb_4$ introduced in Section \ref{subsec-pb4} above. Recall that $U^c$ stands
for the complement $M \setminus U$ of a subset $U \subset M$.

\medskip
\noindent
\begin{prop}\label{thm-pb4-1}
Let $\cU=\{U_1,U_2\}$ and $\cV=\{V_1,V_2\}$ be two open covers of a closed
symplectic manifold $(M,\omega)$.
Then
\begin{equation}\label{eq-pb-simple}
pb(\cU,\cV) = 4pb_4(U_1^c,U_2^c,V_1^c,V_2^c)\;.
\end{equation}
\end{prop}

\medskip
\noindent\begin{proof}

\medskip
\noindent
1) Consider two partitions of unity $f,1-f$ and $g,1-g$ denoted
by $S_f$ and $S_g$ respectively. Observe that they are Poisson commutative and hence
 $\nu_c(S_f)=\nu_c(S_g)=0$. We claim that
\begin{equation}\label{eq-simplepb}
\nu_c(S_f,S_g) =  4||\{f,g\}||\;.
\end{equation}
Indeed,
$$\{x_1f+x_2(1-f),y_1g+y_2(1-g)\}= (x_1-x_2)(y_1-y_2) \{f,g\}\;.$$
Maximizing over $x_i,y_j \in [-1,1]$, we get \eqref{eq-simplepb}.

\medskip
\noindent
2)Let $\cU =\{U_1,U_2\}$ and $\cV=\{V_1,V_2\}$ be two open covers of $M$.
 Consider any two partitions of unity $f,1-f$ and $g,1-g$
subordinated to $\cU$ and $\cV$ respectively. Denote them
by $S_f$ and $S_g$. Since $\text{supp}(f) \subset U_1$,
$$U_1^c \subset W:= M \setminus \text{supp}(f)\;.$$
Note that $f$ vanishes on $W$ and hence $f$ vanishes near $U_1^c$.
Similarly, $f=1$ near $U_2^c$, $g=0$ near $V_1^c$ and
$g=1$ near $V_2^c$. Therefore by \eqref{eq-simplepb}
\begin{equation}\label{eq-upper}
pb(\cU,\cV) \geq 4p\;,
\end{equation}
where
$ p:= pb_4(U_1^c,U_2^c,V_1^c,V_2^c)$.

\medskip
\noindent
3) Consider any pair of functions $f,g: M \to [0;1]$ so that
$f=0$ near $U_1^c$, $f=1$ near $U_2^c$, $g=0$ near $V_1^c$ and
$g=1$ near $V_2^c$. In particular, $f$ vanishes on some open set $X$ containing $U_1^c$.
Therefore the open subset $\{f \neq 0\}$ is contained in a closed subset $X^c$. Thus
$$\text{supp}(f) \subset X^c \subset U_1\;.$$ Applying the same argument to the functions
$1-f$,$g$ and $1-g$ we get that $f,1-f$ and $g,1-g$ are partitions of unity
subordinated to covers $\cU$ and $\cV$ respectively. By \eqref{eq-simplepb}, we have that
$4p \geq pb(\cU,\cV)$. Together with \eqref{eq-upper} this yields $pb(\cU,\cV)=4p$,
as required.
\end{proof}

\subsection{ Proof of Theorem \ref{thm-overlap-layer-pb4}}
Let $\cU= \{U_1,...,U_L\}$ and $\cV=\{V_1,...,V_N\}$ be open covers
of $M$. Fix non-empty subsets $I \subset \Omega_L=\{1,...,L\}$ and $J \subset \Omega_N$.
We shall assume that their complements $I^c$ and $J^c$ are non-empty as well.
Take arbitrary partitions of unity
$\f=\{f_1,...,f_L\}$ and $\g=\{g_1,...,g_N\}$ subordinated to $\cU$ and $\cV$ respectively.
Put $\phi:= \sum_{\alpha \in I} f_{\alpha}$ and $\psi := \sum_{\alpha \in J} g_{\alpha}$.
Then $1 -\phi =  \sum_{\beta \in I^c} f_{\beta}$ and
$1-\psi =  \sum_{\beta \in J^c} g_{\beta}$.

Observe that
$\phi=0$ near $U(I)$, $1-\phi=0$ near $U(I^c)$, $\psi=0$ near $V(J)$ and $1-\psi=0$ near
$V(J^c)$. By definitions of $\nu_c$ and $pb_4$,
$$\nu_c(\f) \geq ||\{ \phi -(1-\phi),\psi-(1-\psi)\}|| =4||\{\phi,\psi\}||$$ $$\geq 4pb_4(U(I),U(I^c),V(J),V(J^c))\;.$$
Since this holds for every partition of unity $\f$ subordinated to $\cU$,
we get \eqref{eq-mu-pb4-overlap}.
\qed

\section{Back to quantum mechanics}\label{sec-backto-quantum}

\subsection{Phase space discretization: a noise-localization uncertainty relation}\label{subsec-noise-local}

Let $(M,\omega)$ be a quantizable closed symplectic manifold. Let $\cU$ be
either a $(d,p)$-regular cover, or a degree $\leq d$ Liouville cover. In the latter case assume
 that $\pi_2(M)=0$. Write $\cA$ for the magnitude of localization of $\cU$. Combining Theorem \ref{thm-noisebound}(i) with \eqref{eq-mag-loc}
we get the following lower bound on the noise indicator of $\cU$:
\begin{equation} \label{eq-mu-fine-displ}
\mu(\cU) \geq C\cA^{-1}\;,
\end{equation}
where the constant $C$ depends only on $d$ and $p$ (in the case of Liouville covers, only on $d$).

This inequality can be spelled out as follows.
Fix a scheme of the Berezin-Toeplitz quantization
$T_m: C^{\infty}(M) \to \cL(H_m)$. For any partition of unity $\f= \{f_j\}$ of $M$ subordinated
to $\cU$ consider the POVM $A^{(m)} = \{T_m(f_j)\}$ corresponding to the quantum registration problem associated to $\cU$. Recalling the definition of the noise indicator and the equality $\hbar=1/m$
we get the following {\it noise-localization uncertainty relation} stated in Theorem \ref{thm-intro-1}:
\begin{equation}\label{noise-localization-single}
\cN_{in}(A^{(m)}) \cdot \cA \geq C \hbar\;\;\; \forall m \geq m_0\;,
\end{equation}
where the number $m_0$ depends on the Berezin-Toeplitz quantization $T$ and the partition of unity $\{f_j\}$.
In other words, a sufficiently fine phase-space localization of the system yields inherent quantum noise.

\subsection{Overlap induced noise}\label{subsec-overlap-concusion}
Let $\cU=\{U_1,...,U_L\}$ and $\cV=\{V_1,...,V_N\}$ be two finite open covers
of a closed quantizable symplectic manifold $(M,\omega)$. Consider POVMs $A^{(m)}=\{T_m(f_i)\}$ and $B^{(m)}= \{T_m(g_j)\}$, where $\f= \{f_i\}$ and $\g= \{g_j\}$ are partitions of unity on $M$, associated
to a pair of registration procedures. Let $C^{(m)}$ be any joint observable of $A^{(m)}$ and $B^{(m)}$
which exists by Proposition \ref{prop-joint-BT}. Combining Theorems  \ref{thm-noisebound}(ii) and
Theorem \ref{thm-overlap-layer-pb4} we get that for all
$I \subset \Omega_L$ and $J \subset \Omega_N$
$$\cN_{in}(C^{(m)}) \geq 2pb_4(U(I),U(I^c),V(J),V(J^c)) \cdot \hbar $$
for all sufficiently large quantum numbers $m$. A similar statement holds true in the case of
a single cover, see Theorem \ref{thm-intro-2} of the introduction which
is an immediate consequence of
Theorem \ref{thm-noisebound}(i), equality  \eqref{eq-pb-single-joint} and Theorem \ref{thm-overlap-layer-pb4}.

Interestingly enough, in certain situations geometry of overlaps provides better bounds for the inherent noise than the ones
coming from the phase space localization. Let us illustrate this in the context of joint measurements
of POVMs associated to a pair of registration procedures. Let $(M,\omega)$ be a quantizable
closed symplectic manifold. Consider a pair of open covers consisting
of exactly two sets, $\cU=\{U_1,U_2\}$ and $\cV=\{V_1,V_2\}$. The corresponding partitions
of unity necessarily have the form $f,1-f$ and $g,1-g$. The Berezin-Toeplitz quantization takes
these partitions into commutative POVMs
$T_m(f), \id-T_m(f)$ and $T_m(g), \id -T_m(g)$ on the two-point space $\Omega_2$.
Each of them is a {\it simple} observable in the terminology of \cite{Busch-2}.
We refer to \cite{Busch-2} for various results on joint measurements of a pair of simple observables.
Applying inequality \eqref{eq-mu-pb4-overlap} with $I=J=\{1\} \subset \Omega_2$, we get that
for the two-set covers
\begin{equation}\label{eq-mu-simple}
\mu(\cU,\cV) \geq 2pb_4(U_1^c,U_2^c,V_1^c,V_2^c)\;.
\end{equation}

\medskip
\noindent \begin{exam}\label{exam-pb4-quadrilateral-1}{\rm
Consider the unit sphere $S^2$ equipped with the standard area form. Consider
a pair of two-set covers $\cU$ and $\cV$ such that $U_1,U_2,V_1,V_2$ are topological discs
with smooth boundaries and the overlap layers $\Lambda (U):= U_1 \cap U_2$ and
$\Lambda (V):= V_1 \cap V_2$ are annuli with the boundaries
$$\partial \Lambda (U) = \partial U_1 \sqcup \partial U_2,\;\;\partial \Lambda (V) = \partial V_1 \sqcup \partial V_2\;.$$
Assume that the intersection $\Lambda (U) \cap \Lambda (V)$ contains a  quadrilateral $\Pi$ whose
edges (in cyclic order) $u_1,v_1,u_2,v_2$ lie on $R_1:= \partial U_1, S_1:= \partial V_1, R_2:= \partial U_2,
S_2:= \partial V_2$ respectively. This is illustrated on Fig. \ref{fig-2}, where the disc $U_1$ lies
to the right of the circle $R_1$, $U_2$ to the left of $R_2$, $V_1$ above $S_1$ and $V_2$ below $S_2$.
The overlap layers $U_1 \cap U_2$ and $V_1 \cap V_2$ are gray while the quadrilateral $\Pi$ is black.
\begin{figure}[tbp]
\centering
\epsfig{file=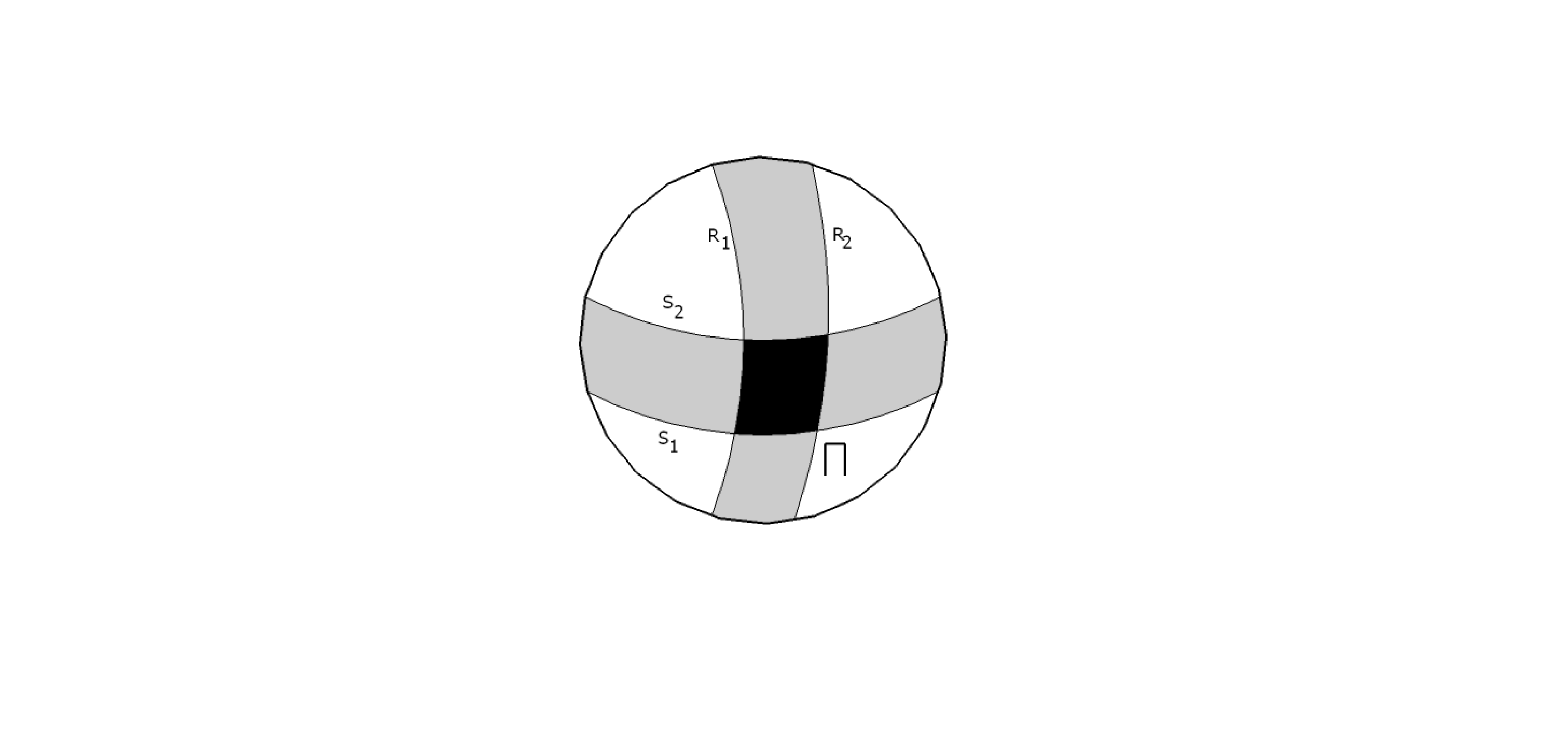,width=1.1\linewidth,clip=}
\caption{A pair of covers of the sphere by two discs.}
\label{fig-2}
\end{figure}
Observe that by \eqref{eq-pb4-monotone} and \eqref{eq-pb4-quadrilateral}
$$pb_4(U_1^c,U_2^c,V_1^c,V_2^c) \geq pb_4(u_1,u_2,v_1,v_2) \geq \frac{1}{\text{Area}(\Pi)}\;,$$
and hence by \eqref{eq-mu-simple}
\begin{equation}\label{eq-twoannuli}
\mu(\cU,\cV) \geq \frac{2}{\text{Area}(\Pi)}\;.
\end{equation}
Another route which could lead to the lower bound on the noise indicator $\mu(\cU,\cV)$ is as follows:
Recall that by Proposition \ref{pb-intersection} $pb(\cU,\cV) \geq pb(\cU \cdot \cV)/4$,
and hence
\begin{equation} \label{eq-ucdotv}
\mu(\cU,\cV) \geq \frac{1}{8}\cdot pb(\cU \cdot \cV)\;.
\end{equation}
For instance, in the situation sketched on Figure \ref{fig-2} the cover $\cU \cdot \cV$ consists of four discs. Denote by $\cA$ the maximal area of these discs which can be interpreted as the magnitude of localization.
If $\cA < \frac{1}{2}\cdot \text{Area}(S^2)$, these discs are displaceable and hence Theorem \ref{thm-main-fine covers}(i) together with \eqref{eq-ucdotv} would yield
$$\mu(\cU,\cV) \geq C \cdot \cA^{-1}\;,$$ where $C$ is a numerical constant.
This bound on the noise indicator is worse than \eqref{eq-twoannuli} when $\text{Area}(\Pi) \ll \cA$,
or, in other words, when the ``characteristic scale" appearing in the relative geometry of the overlap layers
$U_1 \cap U_2$ and $V_1 \cap V_2$ is much smaller then the magnitude of localization.
}\end{exam}

\subsection{Quantum registration and approximate measurements}\label{sec-appendix}
We start this section with some preliminaries on approximate measurements.
Denote by  $\cS$ the set of (not necessarily pure) quantum states $\rho \in \cL(H)$, $\rho \geq 0$, $\text{trace}(\rho)=1$. As usually, a pure state $[\xi] \in \mathbb{P} (H)$ is identified with the
projector $P \in \cS$ to the line $[\xi]$.  For a POVM $B$ on $\R$ and a state $\rho \in \cS$ we denote by $\rho_B$
the corresponding probability measure on $\R$: $$\rho_B(X) = \text{trace}(\rho\cdot B(X))\;,$$
for every Borel subset $X \subset \R$.

Let $A$ be a POVM on $\R$ and $E$ be a projector valued measure on $\R$. Roughly speaking,
$A$ is an approximation to $E$ if for every state $\rho \in \cS$ the probability distribution
$\rho_A$ is ``close" to the probability distribution $\rho_E$. The importance of this notion
is due to the fact that in various interesting situations, pairs of projector valued measures
become jointly measurable only after a suitable approximation (see \cite{Busch-3,BHL2}).

There exist several mathematical ways to formalize the notion of an approximate measurement. In what follows
we stick to the one proposed in \cite{Busch-3} which is based on {\it error bar widths}. Here is the precise
definition (which is slightly simplified in comparison to \cite{Busch-3} since we are dealing
with POVMs on $\R$ supported in a finite set of points, cf. \cite{Miyadera-2}):
Fix $\epsilon \in (0;1)$. Given $x \in \R$ and $\delta>0$ we write $J_{x,\delta}$ for the segment $[x-\delta/2,x+\delta/2]$.
Denote by $\theta$ the infimum of the set of all $w>0$ with the following property:
 for every $x \in \R$ and $\rho \in \cS$
$$\rho_E(x)=1 \Rightarrow \rho_A(J_{x,w}) \geq 1-\epsilon\;.$$
We shall say that {\it $A$ is an $\epsilon$-approximation to $E$ with the error bar width $\theta$.}

 Let us revisit Example \ref{exam-one-func} above: Take a smooth function $F: M \to \R$. Consider the closed interval $I:= [\min F,\max F]$. Let $W_1,...,W_N$ be a cover of $I$ by $N$ open intervals. The sets
$U_k:= \{F^{-1}(W_k)\}$ form an open cover $\cU$ of $M$. Let $\f=\{f_1,...,f_N\}$ be any partition of unity
subordinated to $\cU$.   Applying the Berezin-Toeplitz quantization, we get a sequence of $\cL(H_m)$-valued POVMs $A^{(m)}= \{A^{(m)}_k\}_{k=1,...,N}$  on $\Omega_N$,
 where $A^{(m)}_k =  T_m (f_k)$.

 Assume that the system is prepared in a state $\rho \in \cS$. The quantity
 $$p_k := \text{trace}(\rho \cdot T_m(f_k))$$
 can be interpreted as the probability of the event that the system is registered in the set $F^{-1}(W_k)$
 or, in other words, that the value of the observable $F$ lies in the interval $W_k$.

 In fact, we shall show now that {\it POVM $A^{(m)}$ provides an approximate measurement of the quantum observable $E^{(m)}:= T_m(F)$ corresponding to $F$.} To make this statement precise, we transform $E^{(m)}$ and $A^{(m)}$ into POVMs on $\R$ as follows. In the case of $E^{(m)}$ we use the standard procedure from \cite{Busch}: Let $E^{(m)} = \sum \lambda_j P_j$ be the spectral decomposition of
$E^{(m)}$. Put $\widehat{E}^{(m)}:= \sum P_j \delta_{\lambda_j}$ and consider it as an $\cL(H_m)$-valued POVM on $\R$ representing the von Neumann observable $E^{(m)}$. In case of $A^{(m)}$ we choose in an arbitrary way points $x_k \in W_k$ and put $\widehat{A}^{(m)} = \sum A^{(m)}_k \delta_{x_k}$.

\medskip
\noindent
\begin{thm}\label{thm-approx} Suppose that the length of each interval $W_k$, $k=1,...,N$ is less than $c$.
 Then there exists a sequence of positive numbers $\epsilon_m \to 0$ as $m \to \infty$ so that the POVM $\widehat{A}^{(m)}$ is an $\epsilon_m$-approximation to the sharp observable $\widehat{E}^{(m)}$ with the error bar width $\theta$ which satisfies
 \begin{equation}\label{eq-theta-twosided}
 \frac{\max F -\min F}{4N} \leq \theta \leq 4c
 \end{equation}
 for all sufficiently large $m$.
 \end{thm}

\medskip
\noindent
\begin{rem}{\rm  Let $F,G$ be a pair of smooth functions on $M$. Look at the corresponding operators
$T_m(F)$ and $T_m(G)$ considered as (von Neumann) quantum observables. In general, these two observables
are not jointly measurable. Theorem \ref{thm-approx} combined with Proposition \ref{prop-joint-BT}
above show that {\it they become jointly measurable after a suitable approximation}. The precision
of the approximation tends to zero in the classical limit, while the error bar width remains bounded.
}\end{rem}

\begin{proof}[Proof of Theorem \ref{thm-approx}]

\medskip
\noindent
{\sc Step 1:} First of all we claim that for every continuous function $h$ on $\R$
\begin{equation}\label{eq-BTnew}
|| T_m (h \circ F) - h(T_m(F))||_{op} \to 0\;,\;\;\text{as}\;\;\;m \to \infty\;.
\end{equation}
Indeed, \eqref{eq-BTnew} holds for polynomials by property (BT5) of the Berezin-Toeplitz
quantization. Approximating $h$ on $[-||F||-1,||F||+1]$ in the uniform norm by a polynomial
and applying (BT3), we get \eqref{eq-BTnew}.

\medskip
\noindent {\sc Step 2:} Note that $J_{x,c} \cap W_j = \emptyset$ provided
$x_j \notin J_{x,4c}$. For such a $j$ the function $f_j$ vanishes on the
set $F^{-1}(J_{x,c})$ since it is supported in $F^{-1}(W_j)$. Therefore
\begin{equation}\label{eq-approx-vsp-1}
\sum_{j: x_j \in J_{x,4c}} f_j(y)=1\;\;\;\forall y \in F^{-1}(J_{x,c})\;.
\end{equation}

\medskip
\noindent
{\sc Step 3:} Define a non-negative smooth cut off function $\psi: \R \to [0,1]$ so that
$\psi(s) =1$ when $|s| \leq c/8$ and $\psi(s) =0$ when $|s|\geq c/4$. For $x \in \R$
put $\psi_x(s) = \psi(x-s)$. We claim that
\begin{equation}\label{eq-approx-vsp-2}
\sum_{j: x_j \in J_{x,4c}} f_j \geq \psi_x \circ F\;.
\end{equation}
Indeed, on $F^{-1}(J_{x,c})$ the inequality holds by \eqref{eq-approx-vsp-1}, and  outside
$F^{-1}(J_{x,c})$ the inequality holds since $\psi_x \circ F$ vanishes on this set.

\medskip
\noindent
{\sc Step 4:} For an interval $J \subset \R$ denote by $I_J$ its indicator function.
Observe that
$$\widehat{A}^{(m)}(J_{x,4c}) = \sum_{j: x_j \in J_{x,4c}} T_m(f_j) = T_m\Big{(}\sum_{j: x_j \in J_{x,4c}}f_j\Big{)}\;.$$
By \eqref{eq-approx-vsp-2},
$$ T_m\Big{(}\sum_{j: x_j \in J_{x,4c}}f_j\Big{)} \geq T_m(\psi_x \circ F)\;.$$
By \eqref{eq-BTnew}
$$
T_m(\psi_x \circ F) = \psi_x(T_m(F)) + o(1) \geq I_{J_{x,c/8}} (T_m(F)) + o(1)\;.
$$
It follows that there exists a sequence of positive numbers $\epsilon_m \to 0$ so that
\begin{equation}\label{eq-approx-vsp-3}
\widehat{A}^{(m)}(J_{x,4c}) \geq I_{J_{x,c/8}} (T_m(F))-\epsilon_m \cdot \id\;.
\end{equation}

\medskip
\noindent
{\sc Step 5:} Recall that we denoted $E^{(m)}=T_m(F)$ and  $\widehat{E}^{(m)}:= \sum P_j \delta_{\lambda_j}$,
where $E^{(m)} = \sum \lambda_j P_j$ is the spectral decomposition of
$E^{(m)}$. Assume that $\rho_{\widehat{E}^{(m)}}(x)=1$ for some state $\rho \in S$ and a point $x \in \R$.
This means that
$$\text{trace}( \rho \cdot \sum_{j: \lambda_j = x} P_j)=1\;.$$
Since $\text{trace}(\rho \cdot P_i) \geq 0$ for all $i$ and
$$ \text{trace}(\rho \cdot \sum_i P_i) = \text{trace}(\rho)=1\;,$$
it follows that
$$\text{trace} (\rho \cdot  I_{J_{x,c/8}} (T_m(F)) =1\;.$$
By \eqref{eq-approx-vsp-3} we see that
$$\text{trace} (\rho \cdot \widehat{A}^{(m)}(J_{x,4c})) \geq 1-\epsilon_m\;.$$
By definition, this means that $ \widehat{A}^{(m)}$ is an $\epsilon_m$-approximation
of $ \widehat{E}^{(m)}$ with the error bar width
\begin{equation}\label{eq-theta-twosided-1}
\theta \leq 4c\;.
\end{equation}

\medskip
\noindent
{\sc Step 6:} Let $J \subset (\min F,\max F)$ be a closed interval.
We claim that $T_m(F)$ has an eigenvalue in $J$ for all sufficiently large
$m$. Indeed, let $J=J_{x,w}$.  Define a non-negative cut off function $\phi: \R \to [0,1]$ so that
$\phi(s) =1$ when $|s-x| \leq w/8$ and $\phi(s) =0$ when $|s|\geq w/4$. Assume on the contrary that $T_m(F)$ has no eigenvalues in $J$. Then $\phi(T_m(F))=0$. By Step 1,
$\phi(T_m(F))=T_m(\phi(F))+o(1)$. By property (BT3) of the Berezin-Toeplitz quantization,
$||T_m(\phi(F))||_{op}= 1 + o(1)$. This contradiction proves the claim.

\medskip
\noindent
{\sc Step 7:} Put $\tau=\max F-\min F$.
There exists an interval $$J_{y,w} \subset (\min F,\max F)$$ with $w = \tau/4N$ which does not contain any of the points $x_j$, $j=1,...,N$. By Step 6, $\widehat{E}^{(m)}=T_m(F)$ has an eigenvalue, say, $x$ in $J_{y,w}$ for all sufficiently large $m$. Let $\rho \subset \cS$ be
the projector to one of the corresponding eigenvectors. Then $\rho_{\widehat{E}^{(m)}}(x)=1$,
while $\rho_{\widehat{A}^{(m)}}(J_{x,w})=0$. Thus the error bar width $\theta$ is $\geq w= \tau/4N$.
Together with \eqref{eq-theta-twosided-1}, this completes the proof.
\end{proof}

\subsection{Approximate joint measurements for two components of spin}\label{sec-discussion}
Let us apply the results above to an approximate joint measurement of two components of spin.
In the classical limit, the phase space of the quantum spin system is the two-dimensional sphere
$S^2= \{q_1^2 +q_2^2 +q_3^2 =1\} \subset \R^3$ equipped with the standard area form.

In what follows, we write $c_1,c_2,...$ for positive numerical
constants.
The length of an interval $W \subset \R$ is denoted by $|W|$. Consider an open cover $W_1,...,W_N$ of $[-1;1]$ so that all $W_i$ are connected intervals of the length $ \leq c_1N^{-1}$,
$-1 \in W_1$, $1 \in W_N$ and $ W_i \cap W_j = \emptyset$ for $|i-j| \geq 2$.
Put $U_i= \{q \in S^2\;:\; q_1 \in W_i\}$ and $V_j = \{q \in S^2\;:\; q_2 \in W_j\}$,
and consider the covers $\cU = \{U_i\}$ and $\cV=\{V_j\}$ of the sphere.

\medskip
\noindent
\begin{thm}\label{thm-noise-bound-spin}
\begin{equation}\label{eq-mu-spin}
\mu(\cU,\cV) \geq c_2 N^2\;
\end{equation}
for all $N \geq N_0$, where $N_0$ and $c_2$ depend only on $c_1$.
\end{thm}

\begin{proof} Choose minimal $k$ such that $0 \in W_k$. Consider the subset
$I = \{1,...,k\}$ of $\Omega_N = \{1,...,N\}$. The corresponding overlap layers
(see \eqref{eq-lambdaI} above) are given by
$\Lambda(\cU, I) = \{q_1 \in W_k \cap W_{k+1}\}$
and
$\Lambda(\cV, I) = \{q_2 \in W_k \cap W_{k+1}\}$.
They form a pair of spherical annuli intersecting along two quadrilaterals.
Pick one of these quadrilaterals, say $\Pi$.
In view of our choice of $k$ we can assume that for sufficiently large $N$ the quadrilateral $\Pi$ lies in the domain $\{q_3 \geq 1/2\}$.
In this domain, the spherical area form is given by $a(q_1,q_2)dq_1 \wedge dq_2$ with
$a \leq c_3$. Therefore
$$\text{Area}(\Pi) \leq c_3\cdot |W_k \cap W_{k+1}|^2 \leq c_4N^{-2}\;,$$
and thus \eqref{eq-mu-spin} follows from \eqref{eq-twoannuli}.
\end{proof}

\medskip
\noindent This result deserves to be spelled out in more detail.
Let $\{f_i\}$ and $\{g_j\}$ be any partitions of unity subordinated to the covers $\cU$ and $\cV$ respectively. Their quantum counterparts are given by POVMs $\{T_m(f_i)\}$ and $\{T_m(g_j)\}$.  By Theorem \ref{thm-approx} the
POVMs $\{T_m(f_i)\}$ and $\{T_m(g_j)\}$ $\epsilon_m$-approximate the von Neumann observables $T_m(q_1)$ and $T_m(q_2)$ representing the first and the second components of spin respectively with the error bar width
\begin{equation}\label{eq-thetaup}
1/(2N) \leq \theta \leq 4c_1/N\;.
\end{equation}
Here $\epsilon_m \to 0$ in the classical limit $m \to \infty$.
These POVMs are jointly measurable by Proposition \ref{prop-joint-BT}. Let $C^{(m)}= \{C^{(m)}_{ij}\}$, $i,j=1,...,N$ be such a joint observable. If the system is prepared in a pure state $[\xi] \in \mathbb{P} (H_m)$, the probability of the simultaneous registration of $q_1$ in the interval $W_i$ and of $q_2$ in the interval $W_j$ equals $\langle C^{(m)}_{ij}\xi,\xi\rangle$.  Recall that the Planck constant $\hbar = 1/m$. As an immediate consequence of Theorem \ref{thm-noise-bound-spin} and \eqref{eq-thetaup}, we obtain
that
\begin{equation}\label{eq-width-3}
\cN_{in}(C^{(m)})\cdot \theta^2 \geq c_5 \cdot \hbar \;
\end{equation}
for sufficiently large quantum numbers $m = 1/\hbar$.

It is instructive to compare inequality \eqref{eq-width-3} with the universal uncertainty relation for the error bar widths of approximate joint measurements. This relation has been established first
for joint measurements of position and momentum observables in \cite{Busch-3} and later on extended
to general pairs of observables in \cite{Miyadera-2}. For two components of spin the uncertainty relation
should read
\begin{equation}\label{eq-theta2}
\theta^2 \geq c_6 \hbar\;
\end{equation}
for sufficiently large $m$. This inequality could be extracted from Theorem 2(ii) of \cite{Miyadera-2} combined with a calculation from \cite{SR} and the fact that the precision of the approximation $\epsilon_m$ goes
to zero (see Theorem \ref{thm-approx} above). In our context, since $\cN_{in} \leq 1$, inequality
\eqref{eq-width-3} refines the uncertainty principle for error bar widths \eqref{eq-theta2}
for sufficiently large quantum numbers $m$.

\section{Further directions}\label{sec-further-directions}

\subsection{Does regularity of covers matter?}
Let $(M,\omega)$ be a closed symplectic manifold, and let $\cU=\{U_j\}$ be a finite open cover of $M$
by displaceable subsets with the displacement energy $E(U_j) \leq \cA$.

\medskip
\noindent
\begin{question}\label{quest-displ} Is it true that $pb(\cU) \geq C\cdot \cA^{-1}$,
where the constant $C$ depends only on the symplectic manifold $(M,\omega)$?
\end{question}

\medskip
\noindent The positive answer would enable us to get lower bounds on the inherent noise
of a quantum registration procedure without assuming $(d,p)$-regularity of the corresponding cover.
 A naive intuition suggests that the Poisson bracket invariant $pb$ should be bigger for ``irregular" covers, however, at the moment, we have neither a proof nor a counterexample.

\subsection{Noise-localization uncertainty on wave-length scale}\label{subsec-spec}
Let $(M,\omega)$ be the classical phase space (a closed symplectic manifold).
Consider a $(d,p)$-regular cover of $M$ of magnitude of localization $\cA$. In what follows, $d$ and $p$ are fixed, while $\cA \searrow 0$
will play the role of the small parameter.
Fix a scheme of the Berezin-Toeplitz quantization $T_m: C^{\infty}(M) \to \cL(H_m)$.
For any partition of unity $\f= \{f_i\}$ of $M$ subordinated to $\cU$
consider the POVM $A^{(m)} = \{T_m(f_j)\}$ corresponding to the quantum registration problem associated
to $\cU$. The noise-localization uncertainty relation \eqref{noise-localization-single}
reads
\begin{equation}\label{noise-localization-single-1}
\cN_{in}(A^{(m)}) \cdot \cA \geq C(d,p)\hbar\;\;\; \forall m \geq m_0\;,
\end{equation}
for all $m \geq m_0$.
Let us emphasize that the limits $\cA \to 0$ and $m \to \infty$ do not commute: we first fix $\cA$, and then
choose a sufficiently large $m_0$ which depends on various data including $\cA$.

Assume for a moment that inequality \eqref{noise-localization-single-1} remains valid when $\cA$ is allowed to depend on $m$. Since $\cN_{in} \leq 1$, this yields $\cA \geq C\hbar$.
\medskip
\noindent \begin{question}\label{question-uncert} Does the noise-localization uncertainty
relation \eqref{noise-localization-single-1} remain valid in
the asymptotic regime when $\cA$ is as small as possible, that is $\cA \sim \hbar$?
\end{question}

\noindent

  In such a regime, the dimensionless inherent noise $\cN_{in}(A^{(m)})$ would be bounded away from zero by a positive numerical constant. Observe that if $\cA \sim \hbar =1/m$, the functions $f_i$  entering the partition of unity become dependent on $m$ as well. Therefore a rigorous justification of the above consideration would inevitably require a version of the correspondence principle for the Toeplitz operators of the form $T_m(f^{(m)})$, where the functions $f^{(m)}$  depend on the quantum number $m$, perhaps in a controlled way (say, they obey certain $m$-dependent derivative bounds.)\footnote{I thank Victor Guillemin and Iosif Polterovich for useful discussions on this issue.} Such a version of the correspondence principle for the Berezin-Toeplitz quantization seems to be unavailable at the moment. It would be interesting to develop it by methods of pseudo-differential calculus of Toeplitz operators \cite{BG}.

\subsection{Approximate joint measurements in higher dimensions}
In Section \ref{sec-discussion} we have studied approximate joint measurements of two components
of spin. We started with a pair of open covers $\cU$ and $\cV$ associated to a discretization of
these components with mesh $\sim 1/N$ and showed the the noise indicator $\mu(\cU,\cV)$ satisfies
\begin{equation}\label{eq-mu-twospins-1}
\mu(\cU,\cV) \geq \text{const}\cdot N^2.
\end{equation}
The main tool used in the proof was an elementary lower
bound for the $pb_4$-invariant on surfaces (see Section \ref{subsec-pb4} above) established in \cite{BEP}.
The technique developed in \cite{BEP}, which involves modern symplectic methods,
enables one to extend these results to higher-dimensional phase spaces.

A meaningful non-trivial example is given by a system of two spins whose phase space is the product $M=S^2 \times S^2$. A point of $M$ is a pair $(q,r)$ where $q,r$ are unit vectors in $\R^3$. The total spin of the composed system equals $q+r$. It follows from \cite{BEP} that inequality \eqref{eq-mu-twospins-1} holds
for approximate joint measurements of two components  $q_1$ and $q_2$ of the spin of the first system.
It would be interesting to extend it to approximate joint measurements of two components $q_1+q_2$ and
$r_1 + r_2$ of the total spin. Even though the geometry of this problem is more involved, it sounds likely that such an extension can be obtained by methods of symplectic field theory following the lines of Section 6 of \cite{BEP}.

\subsection{Localization vs. overlaps?} In the present paper we discussed two methods of detecting inherent noise, phase space localization and geometry of overlaps. The former looks as a more universal one,
while the latter seems to be more efficient in situations where it is applicable at all. For instance,
existence of overlap induced noise is unclear for greedy coverings of higher-dimensional manifolds. At the same time, both methods yield certain lower bounds for the noise indicator in the case of approximate joint measurements of two components of spin, and in this situation geometry of overlaps yields
better results.

Is there a general recipe enabling one to choose between these two methods in specific
examples? The answer to this question is unclear to us, both from the symplectic and the physical viewpoints.
In fact, it would be interesting to understand a physical meaning of the overlap induced noise. For instance,
does it admit an interpretation in terms of some uncertainty relation?

\medskip
\noindent{\bf Acknowledgement:} I am indebted to Paul Busch  for pointing out \cite{Busch-private}  that
the results of my recent paper \cite{P} should extend to the case of joint measurements and for
supplying me with a number of helpful bibliographical references. The present paper is a (somewhat over-sized)
reply to Busch's comment. I thank Strom Borman, Victor Guillemin, Iosif Polterovich, Daniel Rosen
and Sobhan Seyfaddini for useful discussions and comments on the first draft of the paper.

The first draft of this paper had been written during my stay at University of Chicago.
I thank the Department of Mathematics at University of Chicago for the warm hospitality
and an exceptional research atmosphere.

\begin{tabular}{l}
School of Mathematical Sciences\\
Tel Aviv University\\
Tel Aviv 69978, Israel\\
polterov@runbox.com\\
\end{tabular}

\end{document}